\documentclass[a4paper]{article}
\usepackage[a4paper,margin=3cm]{geometry}
\usepackage[utf8]{inputenc}
\usepackage{amsthm,amssymb,amsbsy,amsmath,amsfonts,amssymb,amscd,dsfont}
\usepackage{stmaryrd}
\usepackage{amsmath}
\usepackage{amssymb}
\usepackage{xcolor}
\usepackage{graphicx}
\usepackage{subcaption}
\usepackage{authblk}

\usepackage{subfiles}
\usepackage{csquotes} 
\usepackage[todonotes={textsize=scriptsize}]{changes}
\usepackage[colorlinks=true,linkcolor=blue, citecolor=red]{hyperref}
\usepackage{color}

\usepackage{microtype}
\usepackage{hyperref}

\newtheorem{theorem}{Theorem}[section]
\newtheorem{lemma}[theorem]{Lemma}
\newtheorem{proposition}[theorem]{Proposition}

\theoremstyle{definition}

\newtheorem{assumpt}{Assumption}

\theoremstyle{remark}

\newcommand{\eps}{\varepsilon}
\newcommand{\R}{\mathbb{R}} 
\newcommand\E{\mathbb E}
\renewcommand\r{\mathbb{R}}

\newcommand\pro[1]{\mathbb{P}\left(#1\right)}
\newcommand\esp[1]{\mathbb{E}\left[#1\right]}
\newcommand\uno[1]{\mathds{1}_{\left\{#1\right\}}}

\newcommand\tx{\rho}

\title{Invasion dynamics with vanishing fitness for a quasi-critical birth-death process}

\author[1]{Vincent Bansaye}
\author[2]{Nadia Belmabrouk}
\author[3]{Xavier Erny}
\author[4]{Simon Girel}

\affil[1,2]{CMAP, \'Ecole polytechnique, Institut Polytechnique de Paris, Palaiseau, France}
\affil[1]{MERGE, INRIA, \'Ecole polytechnique, Institut Polytechnique de Paris, Palaiseau, France}
\affil[3]{SAMOVAR, T\'el\'ecom SudParis, Institut Polytechnique de Paris, Palaiseau, France}
\affil[4]{Université Côte d’Azur, CNRS, LJAD, France}

\date{}

\begin{document}
\maketitle 

\vspace{0.25cm}
\begin{abstract}
We study the invasion dynamics of populations exhibiting positive density-dependent effects. We start with a single individual  and consider a  single-type birth and death process. The initial individual growth rate vanishes  but it increases with the population density, proportionally to the number of individuals divided by a scaling parameter $K$. Before reaching the macroscopic scale~$K$, the population process is almost critical. 
 We prove that the probability for the population to reach the macroscopic level $K$ decreases  as $1/\sqrt{K}$ as $K$ goes to infinity. We also describe the associated trajectories and show that invasion can be split into three time periods. First, the process needs to escape from zero, and conditioning on survival, it grows linearly until the order $\sqrt{K}$. The scaled process is approximated by a diffusion,  as for critical branching process, with an additional drift term coming from cooperation, which  breaks the branching property. Second, in intermediate scale $\sqrt{K}$, we observe another diffusion, surviving with positive probability, without conditioning. Finally, beyond $\sqrt{K}$ scale,  the  process can be approximated by a classical macroscopic ODE limit. 
The proof of the first phase involves change of probability and characterization of uniform integrability of martingales, while the two other phases rely on uniform approximations on polynomial time scales.
\end{abstract}
\section{Introduction}

We are interested in the invasion of  populations which enjoy positive density-dependent effect. It corresponds to  cooperation  coming for instance from  mutualism, Allee effect,  or positive loops. 
Let us define \emph{fitness} here as the individual growth rate in small population, i.e. the gap between the birth and death rates for one single individual in a fixed environment.

When this fitness is positive, starting from one individual (or small populations), the process is usually approximated by a supercritical branching process and  survives with a (non vanishing) positive  probability.  We refer to 
\cite{ball1995strong, bansaye2024sharp, barbour2015escape, barbour2013approximating, champagnat2006microscopic, champagnat2011polymorphic} and references therein for such results motivated in particular by epidemics and invasion-fixation in evolution.
When it survives, the population density is approximated by a deterministic ODE, using classical fluid limit theorems \cite{ethier2009markov}. This deterministic approximation actually works as soon as the population becomes large. This includes intermediate scales, when the order of magnitude of the density of the invasive population is between  $1$ and $K$, see \cite{bansaye2024sharp}. When the fitness is negative, the probability to reach macroscopic values is very low, decreasing exponentially with $K$.

In this paper, we are interested in critical regimes, i.e.  vanishing fitness for invaders. Our original motivation comes from invasion of cancer cells, where the cancer-associated mutation might leave unchanged the intrinsic fitness of each isolated cancer cells, that  remains the same as in the resident population. Instead, it induces positive feedback loops among cancer cells, resulting in a density-dependent selective advantage, that vanishes at low cell density.  We focus in this paper on the mutant population and consider 
a toy model, simplifying the study of the invasion. We  can understand and quantify its capacity to invade. We hope that this study, which is of independent interest, give some keys for more complex models of invasion of mutants in competition with a resident population.

More precisely,  letting $K\geq 1$ be the scaling parameter, we consider  a  birth-death process~$N^K$ starting from $N_0^K=1$ a.s. 
The transition rates of this Markov process are given for $n\in \mathbb{N}$ by
\begin{align*}
 & n\to n+1 \text{ with rate } n\,b\left(\frac{n}{K}\right),\\
 &   n\to n-1  \text{ with rate } n\,d\left(\frac{n}{K}\right),
\end{align*}
where $b$ and $d$ are the individual density-dependent birth and death  rates, respectively. We are interested in the case where $b$ and $d$ are very close when the population density is low. 
The typical examples correspond to cooperation impacting the birth rates, where 
$$ b(x)=a\,(x\wedge c) +\tx, \qquad d(x)=\tx,$$
with $a,\tx,c\geq 0$, and to cooperation making the death rate decrease:
$$b(x)=\tx, \quad d(x) = (\tx - a\,x)_+.$$
In these examples, 
when the cooperation coefficient $a=0$, $b(x)=d(x)=\tx$. This yields the classical critical linear birth and death process \cite{athreya2012branching}. It enjoys the branching property and the generating function of the marginal laws are explicit. Assuming that the second moment is finite, the survival probability at time $t$ is equivalent to $1/t$,  up to a multiplicative constant. 
Moreover,  conditionally on survival at time $t$, the process grows linearly during the time window $[0,t]$ and $Z_t/t$ converges to an exponential law. Renormalizing by $t$ and conditioning on survival, the process 
is approximated by the critical Feller diffusion, under a suitable change of probability. We refer to the seminal works
\cite{heathcote1967refinement, hering1977minimal, seneta1966quasi} for details. 
In this framework, reaching macroscopic population sizes $K$ occurs  with probability of order $1/K$. Then it happens at time of order $K$, and the renormalized process is given by a  diffusion. The classical approach for estimating survival probability relies on generating functions and the approximation of  trajectories conditioned on survival is obtained by $h$- transform.  Our approach  leads to an alternative proof for the survival probability, relying on the change of probability associated with the $h$-transform.\\

In this work,  we consider  processes with cooperation, i.e.  $a>0$. They do not enjoy the branching property any longer. Our framework is actually a bit  more general that the the previous examples. It includes any birth and death functions $b$ and $d$ of class $C^2$ such that $b(0)=d(0)=\tx$ and thus recovering classical model for null fitness. 
 More precisely, in the whole paper, we assume 
 \begin{assumpt}\label{assu}
 The functions $b$ and $d$ are locally Lipschitz on $[0,\infty)$ and
 $b-d$ is upperbounded,  i.e. there exists $C>0$ such that for any $x\geq 0$, $b(x)\leq d(x)+C$.
 Moreover there exist $a,\rho,\alpha>0$ such that
 $$\sup_{x\in (0,1]}  \frac{\vert b(x)-d(x)-ax \vert}{x^{1+{\alpha}}}<\infty,\qquad 
 \sup_{x\in (0,1]}  \frac{\vert b(x)+d(x)-2\rho \vert}{x}<\infty.
 $$
\end{assumpt}
The local Lipschitz assumption and the sublinearity of the growth rate $x(b(x)-d(x))$ guarantee non explosion of the birth and death process and allow us to work with a process defined for any time.  They also ensure the existence and uniqueness of the associated ODE, for any time. It could be relaxed by a localization argument but will be enough for our purpose. \\

The cooperation $a>0$ changes both the order of magnitude of the invasion probability and trajectories associated with. We prove that the probability of reaching macroscopic levels is now of order $1/\sqrt{K}$, instead of  $1/K$. Moreover, the process describing invasion exhibits now three phases, corresponding to  different approximations. The first phase requires conditioning on survival, similarly to the critical branching process conditioned on survival, with an additional drift term in the diffusion approximation coming from cooperation. A new intermediate diffusive regime appears, without conditioning, when the process is at scale $\sqrt{K}$. During this phase, the survival probability is positive and less than one. The last phase is the classical deterministic approximation, with probability one for survival. We refer to Lemma~\ref{lem:extinction} for a preliminary result showing the change of survival probability in these different phases. \\

Our main difficulty in this paper is to describe the first phase. Indeed, we need to estimate the survival probability and describe the trajectory, without enjoying the branching property. It is achieved using a change of probability. Such an approach for conditioning a process to survive is classical, but the form of the martingale leads to delicate estimations for proving convergence as $K$ goes to infinity, see in particular exponential estimate in Lemma \ref{estmomexpo}. The process in the first phase is close to a critical branching process. Our issue is thus related to the study of almost or quasi critical branching process, which has already attracted lots of attention for several reasons. For instance, at the population level, fluctuations of large populations around equilibrium are linked to almost critical process, see for instance 
\cite{kurtz1981approximation}. 
The case of processes whose drift becomes asymptotically critical has also attracted a lot of attention, motivated for instance by competitive effect which make the growth rate decline to zero, see e.g. \cite{aspandiiarov1996passage,kersting2017recurrence, klebaner1989linear, sagitov2015skeletons}. Techniques involve Lyapunov functions and change of probability and we refer to 
 \cite{denisov2025markov}
 for an overview on the topics, and to references therein. 
 We also rely on change of probability but the situation we consider here is somewhat different. Indeed, we leave the critical area here, instead of becoming more and more critical. Moreover, the limit we consider is linked to an additional scaling parameter $K$, fixed along time, rather than a long time asymptotic. We prove that the survival probability is of order $1/\sqrt{K}$ and approximate the conditional trajectory by a diffusion, see Theorem \ref{thm:1}. 

The second phase consists of describing the population process $N^K$ when it is of order of magnitude  $\sqrt{K}$. At this stage, it survives with positive probability and no conditioning is required for invasion. On the accelerated time $\sqrt{K}$, the trajectory is approximated by a new diffusion, see Theorem \ref{thm:regime22}. The proofs in the second phase mainly involve semigroup theory. The third regime is more classical and its proofs mainly rely on stochastic calculus. It corresponds to the macroscopic phase and approximation  by the solution of the ODE
$$x'(t) = x(t) (b-d)(x(t)).$$
We refer to Theorem \ref{thm:3} for a precise statement, where approximation is obtained beyond scales  $K^{1/2+\varepsilon}$. 

\begin{figure}[h]
\centering
\includegraphics[width=0.6\textwidth]{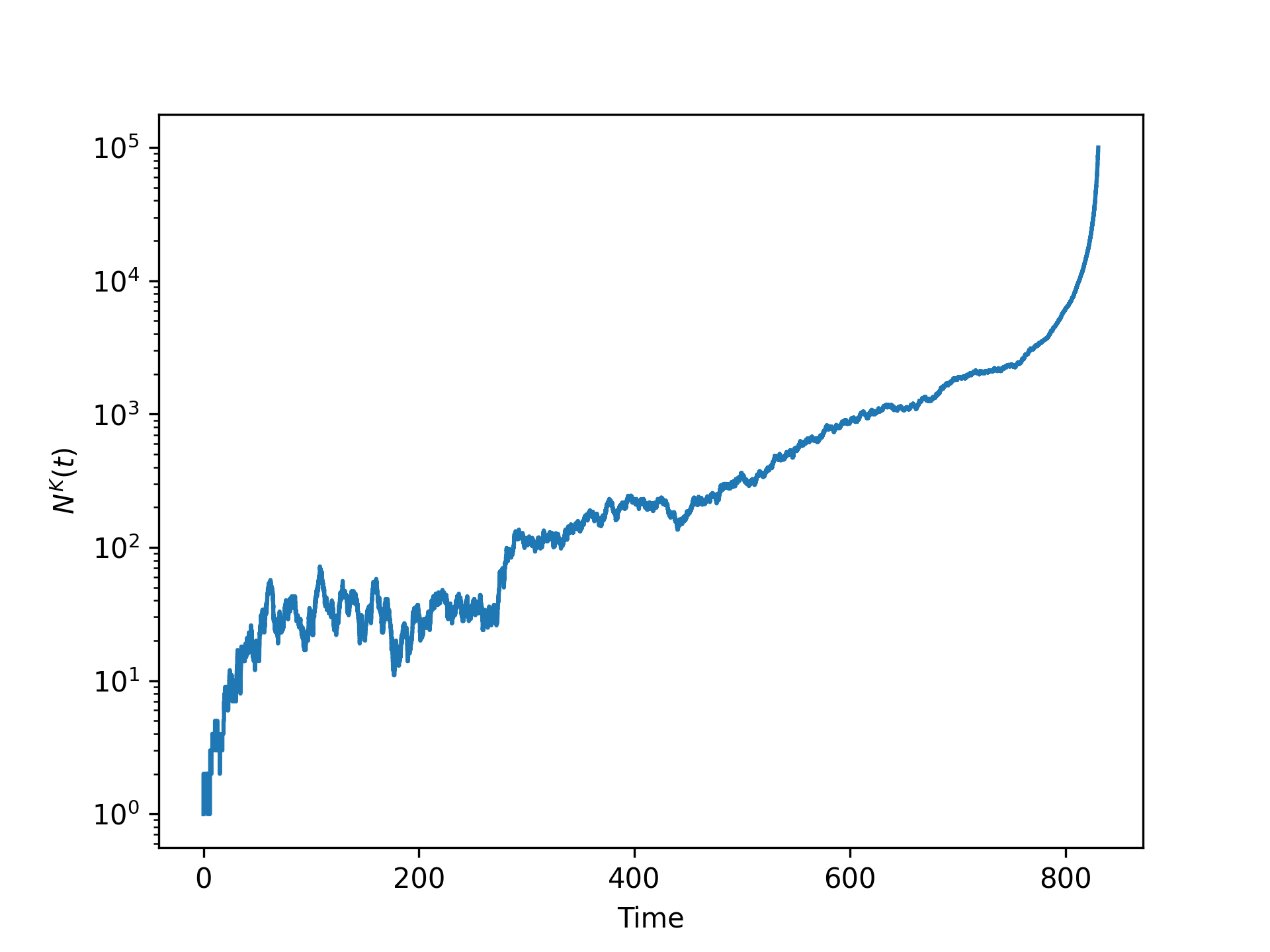}
\caption{Trajectory of $N^K_t$ from $1$ to $K,$ with $b(x)=0.5x+0.3$ and $d(x)=0.3$}
\end{figure}



{\bf Organization of the paper.} In Section~\ref{sec2}, we state our main results concerning each phase of invasion. Section~\ref{sec3} is devoted to the proofs of the main results.

\section{Main results}\label{sec2}
Under Assumption~\ref{assu}, the birth and death process $N^K$ can be represented as the unique strong solution in $\mathbb D([0,\infty), \mathbb R_+)$ of the following SDE (see e.g. Theorem~$IV.9.1$ of \cite{ikeda})
\begin{align}
N^K(t)=& N^K(0)+\int_0^t\int_0^{+\infty}\textbf{1}_{\{u\leq N^K(s^-)b(N^K(s^-)/K)\}} \pi^b(dsdu) \nonumber\\
&
-\int_0^t\int_0^{+\infty}\textbf{1}_{\{u\leq N^K(s^-)d(N^K(s^-)/K)\}}\pi^d(dsdu),\label{SDEdef}
\end{align}
where $\pi^b,\pi^d$ are independent Poisson measures on $\r_+^2$ with Lebesgue intensity. We consider
$$T^K_n=\inf \{t>0~:~N_t^K=n\} \in [0,\infty]$$
to be the hitting time of~$n$ by the process~$N^K$.

Our goal is to approximate formally the process $N^K$ starting from one at $t=0$, until it reaches the macroscopic level $K$ (i.e. at time $T^K_K$), as $K$ goes to infinity. For this purpose, we split the dynamics in three phases, respectively denoted \emph{microscopic}, \emph{intermediate} and \emph{macroscopic} and corresponding respectively to the following population sizes 
$$1\leq N^K_t\leq K^{1/2}; \quad 
K^{1/2}\leq N^K_t\leq K^{1/2+\eps},~~\textrm{for some }\eps>0; \quad K^{1/2+\eps}\leq N^K_t\leq K.$$
Notice that each of the three steps is characterized by the asymptotic behavior of the survival  probability, as $K$ goes to infinity. This is stated in the result below, whose proof is postponed to Appendix~\ref{sec:ext}. 
\begin{lemma}\label{lem:extinction} The following limits hold.
$ $
 \begin{enumerate}
 \item[$(i)$] For any $\beta\in [0,1/2),$
$$\lim_{K\rightarrow \infty} \,  {K^{1/2-\beta}} \,  \mathbb{P}_{\lfloor K^{\beta}\rfloor}\Big(T^K_0 > T^K_{\lfloor \sqrt{K}\rfloor} \Big) = \frac{1}{\int_0^1e^{-\frac{a}{2\tx} x^2}dx}.$$
\item[$(ii)$] $$\lim_{K\rightarrow \infty}\mathbb{P}_{\lfloor\sqrt{K}\rfloor}\Big(T^K_0 > T^K_K\Big)
\in(0,1).$$
\item[$(iii)$] For any $\beta > 1/2,$ $$\lim_{K\rightarrow \infty}\mathbb{P}_{\lfloor K^{\beta}\rfloor}\Big(T^K_0 > T^K_K \Big)
\longrightarrow 1.$$
\end{enumerate}
\end{lemma}

This result reveals the three phases that we study below.


\subsection{Microscopic phase}\label{sec:regime1}


In this regime, we start from one individual, i.e. $N^K_0=1$. In particular, the extinction probability vanishes as $K$ goes to infinity; see Lemma~\ref{lem:extinction}$.(i)$. So we study the law of $N^K$ conditionally to its survival, until reaching scale $\sqrt{K}$. For convenience we consider $K$ such that $\sqrt{K}$ is an integer, and let $K$ go to infinity. However, all the results can be extended by writing $\lfloor \sqrt{K}\rfloor$ instead of $\sqrt{K}$.
In this microscopic phase, the trajectory will be given by the unique strong solution  in $\mathbb D([0,\infty), \mathbb R_+)$ of the following SDE (see e.g. Theorem~$IV.9.1$ of \cite{ikeda}) \begin{equation}\label{eq:w}
W_t \;=\int_0^t(2\tx + aW_s^2)\,ds \;+\; \int_0^{t} \sqrt{2\tx W_s}\,dB_s, \
\end{equation}
for  $t\leq \tau_1=\inf\{t\geq 0 : W_t=1\}$, 
with $B$ a standard one-dimensional Brownian motion, and $W_t=1$ for $t\geq \tau_1$.

\begin{theorem}\label{thm:1}
For any function~$F:\mathbb D([0,\infty),[0,\infty))\rightarrow\mathbb{R}$ continuous and bounded,
\begin{align*}
&\mathbb{E}\!\left[
F\left({K}^{-1/2} \, N^K_{t\sqrt{K}\wedge T^K_{\sqrt K}} \, : \, t\geq 0  \right) {\bf 1}_{ T^K_{\sqrt{K}}< T^K_0}
\right]\\
&\qquad \qquad \qquad \underset{K\rightarrow\infty}{\sim} \frac{1}{\sqrt{K}} \displaystyle\mathbb{E}\!\left[
F\big(W_{t} : t\geq 0\big)\, \,
\exp\!\left(a\int_0^{\tau_1} \!W_s ds\right)
\right].    
\end{align*}
\end{theorem}
Applying this to $F=1$, 
$$\mathbb P_1(T^K_{\sqrt{K}}< T^K_0)\sim_{K\rightarrow \infty} \frac{c}{\sqrt{K}}$$
where, using also  the previous Lemma~\ref{lem:extinction}.$(i)$
$$c=\displaystyle
\mathbb{E}^x\!\left[
\exp\!\left(a\int_0^{{T_1}} \! W_s ds\right)
\right]=\left(\int_0^1e^{-\frac{a}{2\tx} x^2}dx\right)^{-1} \in (0,1).$$
This theorem yields  the trajectory of the birth and death process  conditioned to reach  $\sqrt{K}$ before $0$, and thus survive at least the time needed to reach $\sqrt{K}$.

\begin{figure}[h]
\centering
\includegraphics[width=0.6\textwidth]{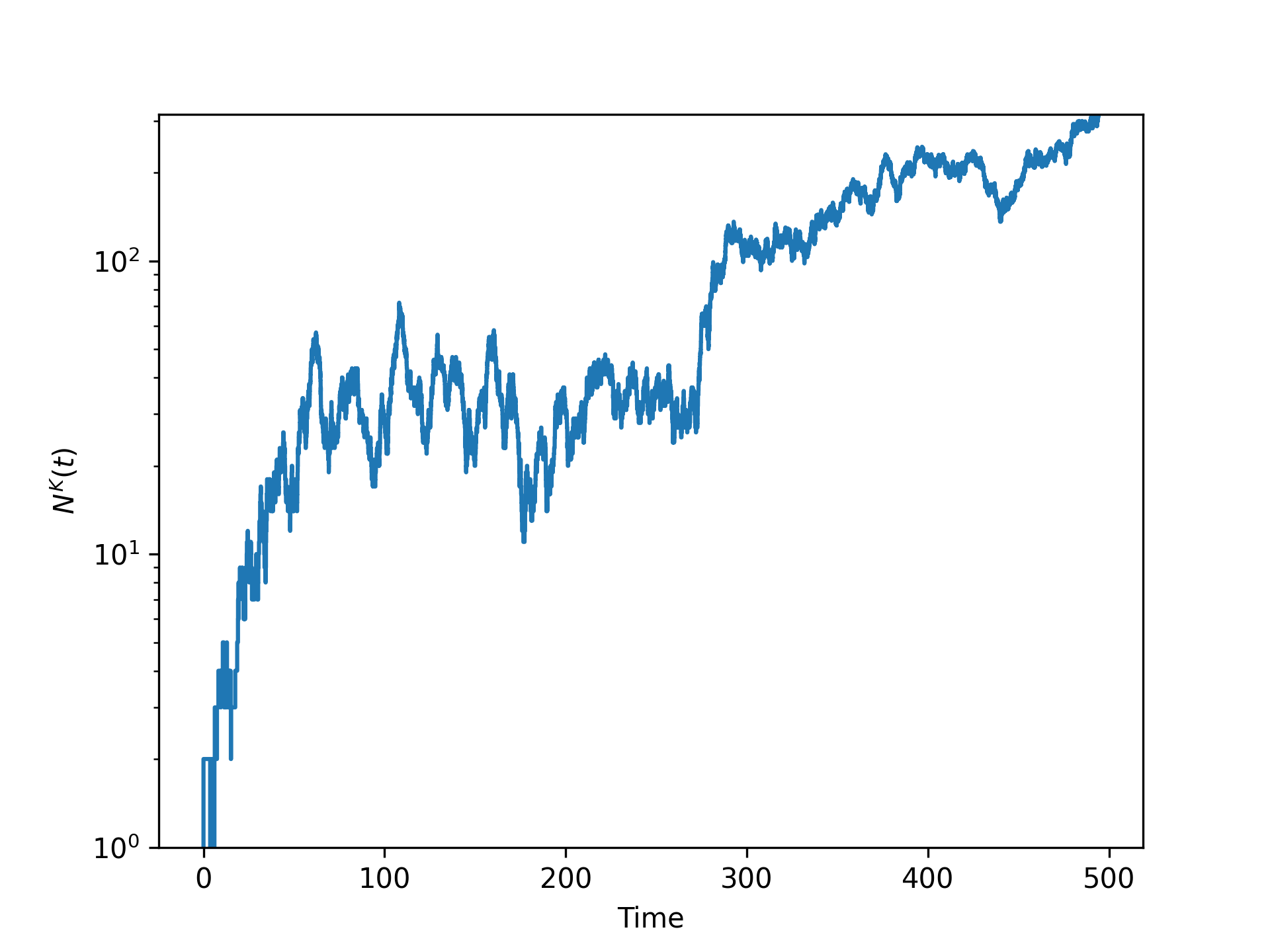}
\caption{Trajectory of $N^K_t$ from $1$ to $\sqrt{K},$ with $b(x)=0.5x+0.3$ and $d(x)=0.3$}
\end{figure}

\subsection{Intermediate phase}\label{sec:regime2}
In  the previous subsection, the process has reached  $K^{1/2}$. We start from $N^K_0=K^{1/2}$.
Let for $t\geq 0$
$$\xi^K_t := K^{-1/2} N^K_{t K^{1/2}}, $$
starting from $1$ and  consider the unique strong solution $\bar\xi$ of the following SDE
$$\bar\xi_t =1+  \int_0^t a\bar \xi_s^2 ds +\int_0^t \sqrt{2\tx \bar \xi_s} dB_s.$$
Observe that the solution may explode in finite time, but will be considered before reaching level $\sqrt{K}^{1+\varepsilon}$.
Recall from Lemma~\ref{lem:extinction}$ (ii)$ that  the process $N^K$ can either go extinct or not with positive probability, as $K$ goes to infinity.
Writing  
$$\tau(v) := \inf\{t>0~:~\bar \xi_t=v\},$$
two events may occur
\begin{itemize}
    \item[$\searrow$] either $N^K$ comes back to microscopic scale, i.e. $\tau(K^{-\eta})<\tau(K^{\varepsilon})$
    and we are back to the first phase,  where the probability of extinction goes to one,
    \item[$\nearrow$] 
or $N^K$ goes beyond scale $\sqrt{K}$,
    i.e. $\tau(K^{-\eta})>\tau(K^{\varepsilon})$
    and we will enter the third regime, leading to macroscopic phase.
\end{itemize}

These two scenarios are illustrated in Figure~\ref{fig:comparison}.
\begin{figure}[htbp]
\centering
\begin{subfigure}{0.6\textwidth}
\centering
\includegraphics[width=\linewidth]{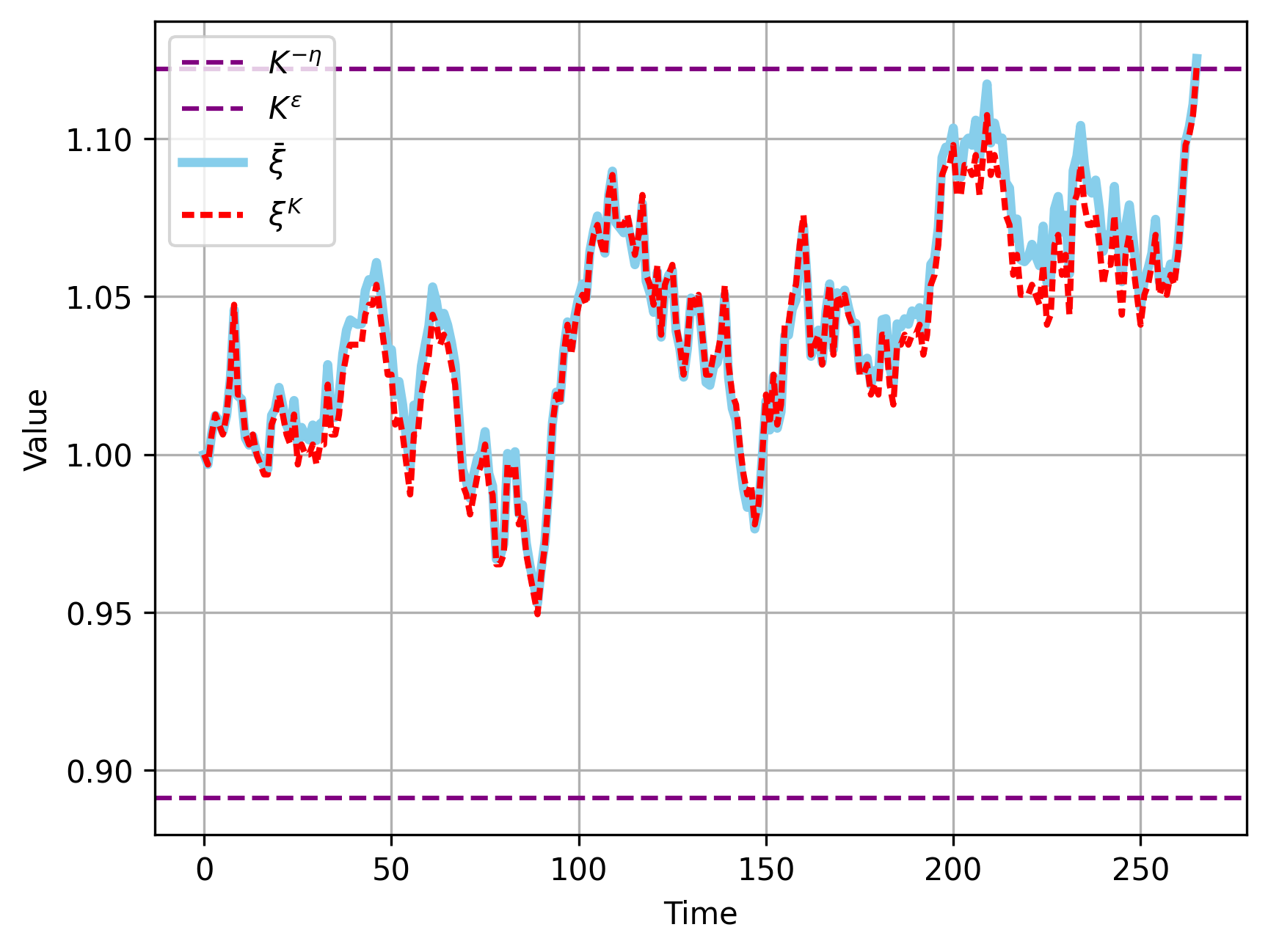}
\label{fig:2}
\end{subfigure}
\hfill

\begin{subfigure}{0.6\textwidth}
\centering
\includegraphics[width=\linewidth]{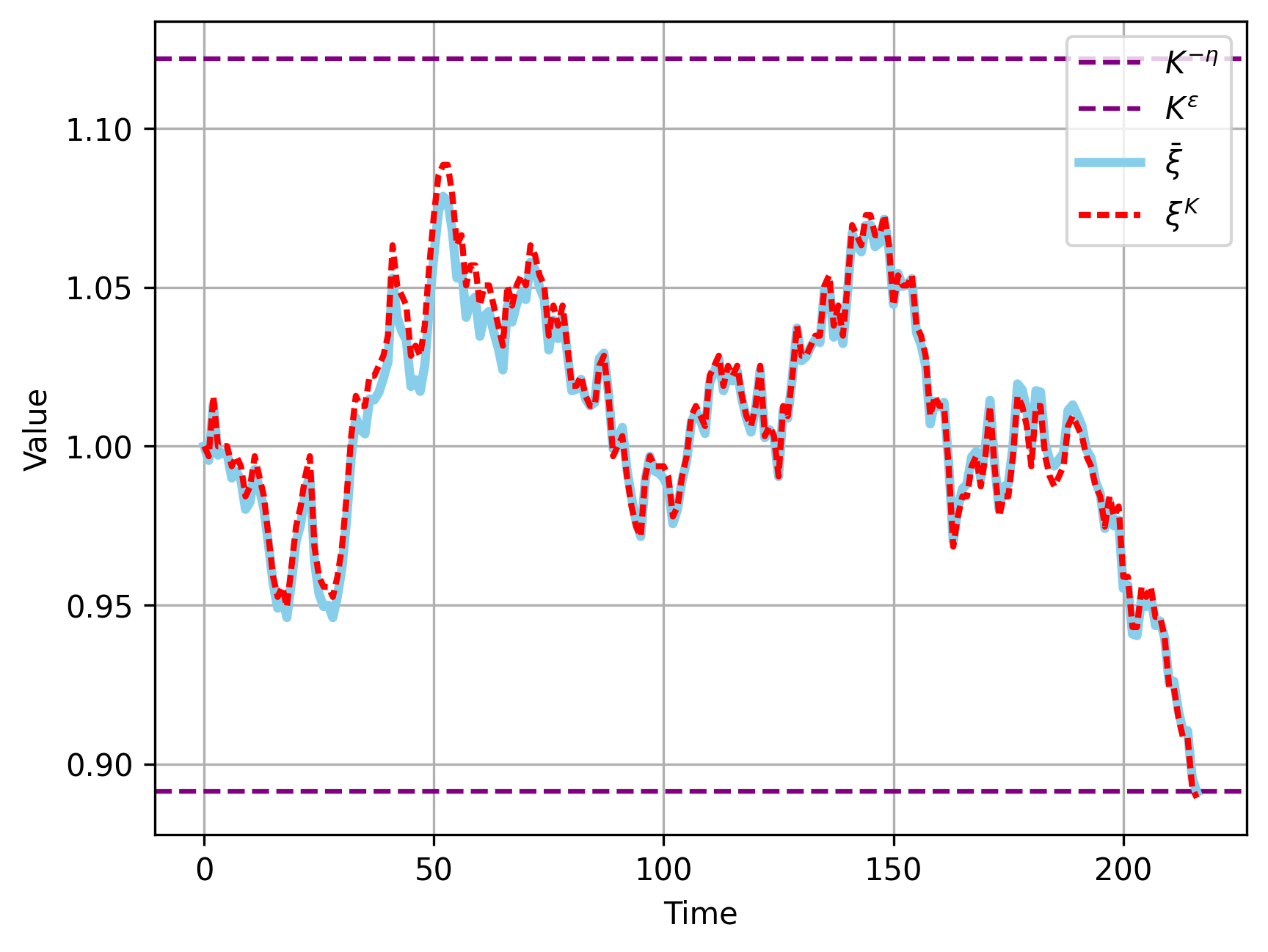}
\label{fig:1}
\end{subfigure}

\caption{Trajectory Coupling Between  $\bar{\xi}$ and $\xi^K$ for $K^{-\eta}\leq\xi^K\leq K^{\eps},$ with $b(x)=0.5x+0.3$, $d(x)=0.3$, $\varepsilon=0.0005$ and $\eta=0.01$}
\label{fig:comparison}
\end{figure}
The approximation result for this regime is formally stated as follows.
\begin{theorem}\label{thm:regime22}
For all $n\in\mathbb{N}^*,g_1,...,g_n\in C^3_b(\r)$, there exist (arbitrarily small) $\eps,\eta>0$ such that the following convergence in probability holds
$$\sup\left|\mathbb{E}_x\left[g_1(\xi^K_{t_1})...g_n(\xi^K_{t_n})\right] - \mathbb{E}_x\left[g_1(\bar\xi_{t_1})...g_n(\bar\xi_{t_n})\right]\right|\underset{K\to\infty}{\longrightarrow}0,$$
 where the $\sup$ is over $(t_1,\ldots  ,t_n) \in [0,\ \tau(K^\eps)\wedge \tau(K^{-\eta})]^n.$
\end{theorem}

\subsection{Macroscopic phase}\label{sec:regime3}

In this last regime, we start from $N^K_0 = K^{\eps + 1/2}$ (for some $0<\eps<1/2$) and consider 
$$X^K_t = \frac{N^K_t}{K}$$
We now compare this stochastic process with the deterministic solution of
$$x'(t) = x(t) (b-d)(x(t)),$$
 and introduce the hitting time of $v$ when this solution starts from $x_0$
$$\tau(x_0, v) := \inf\{t>0~:~x(t)=v\textrm{, with }x(0) = x_0\}.$$

\begin{figure}[h]
\centering
\includegraphics[width=0.6\textwidth]{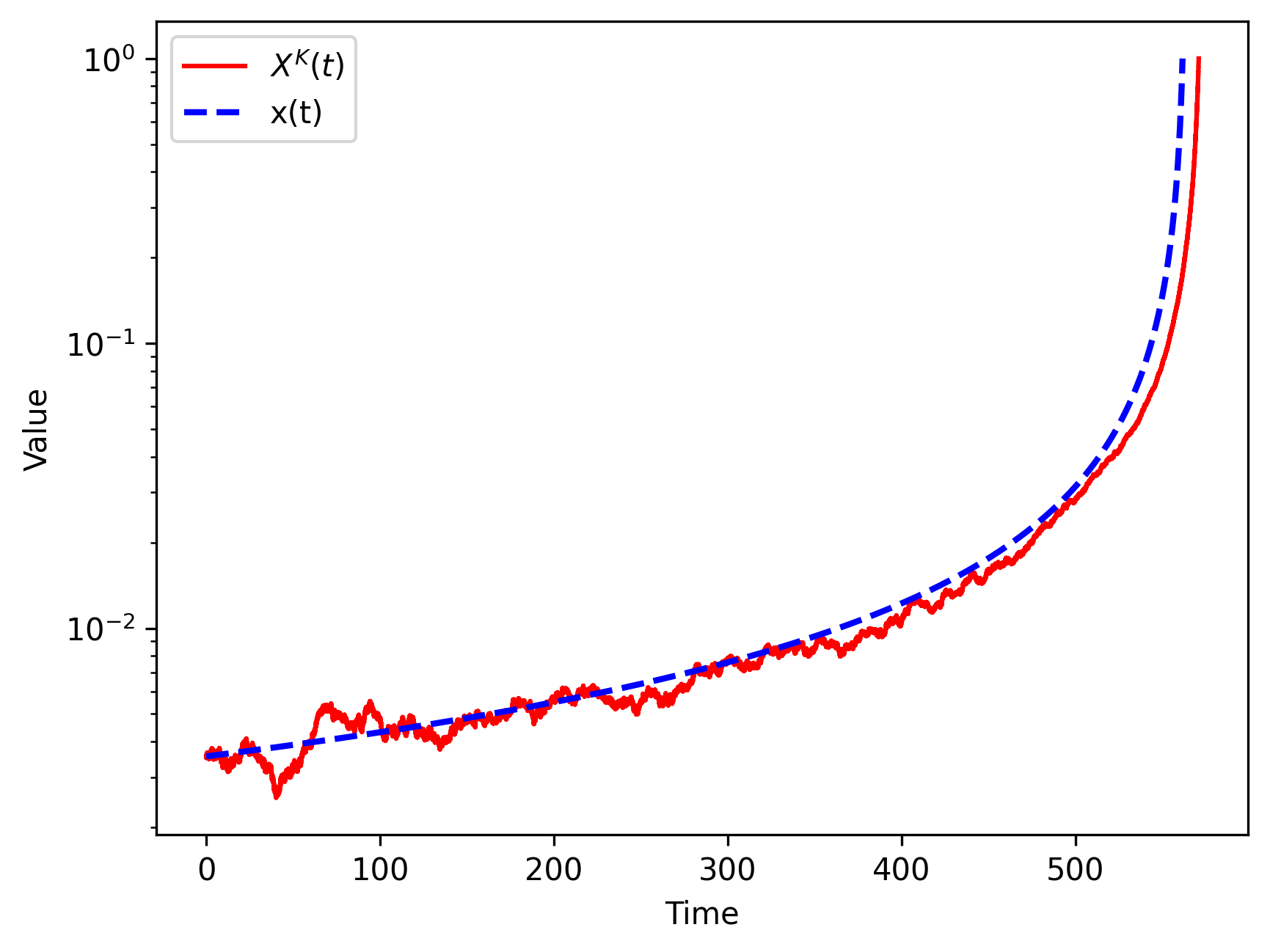}
\caption{Trajectories of $X^K_t$ and $x(t),$ with $b(x)=0.5x+0.3$, $d(x)=0.3$}

\end{figure}

Recalling Assumption~\ref{assu},
 the function $b-d$ is Lipschitz continuous on. Consequently, the function $(x(t))_t$ is well-defined. In addition, as a consequence of Assumption~\ref{assu}, there exists some $v>0$ such that, for $x\leq v,$
$$(b-d)(x)\geq \frac{a}{2}x,$$
hence $\tau(x_0,v)$ is finite for any~$x_0<v$. In the rest of this section, we fix such a~$v$.

\begin{lemma}\label{lem:341}
Provided that
$$X^K(0) = x_K(0) \underset{K\rightarrow\infty}{\sim} \frac1{K^\beta},$$
with $\beta< 1/4$ (in other words, $N^K(0)\sim K^\gamma$ with $\gamma > 3/4$),
$$\underset{t\leq \tau(x_K(0),v)}{\sup}\left|\frac{X^K(t)}{x_K(t)} - 1\right|\overset{\mathbb{P}}{\longrightarrow}0.$$
\end{lemma}
This implies
\begin{theorem}\label{thm:3}
For any~$\beta<1/2$, if
$$X^K(0) = x_K(0) \sim \frac{1}{K^\beta},$$
then
$$\underset{t\leq \tau(x_K(0),v)}{\sup}\left|\frac{X^K(t)}{x_K(t)} - 1\right|\overset{\mathbb{P}}{\longrightarrow}0.$$
\end{theorem}

\section{Proofs}\label{sec3}

\subsection{Study of microscopic phase : proof of Theorem~\ref{thm:1}}
We consider $N_0^K>0$ a.s. and 
 define 
\begin{equation}\label{eq:MK}
M^K_t := \frac{N^K_t}{N^K_0} \exp\left[-\int_0^t  \left(b\left(\frac{N^K_s}{K}\right)-d\left(\frac{N^K_s}{K}\right)\right)ds\right].
\end{equation}
We denote
$$T^K = T^K_{\sqrt{K}}\wedge T^K_0=\inf\{t\geq 0 : N_t^K\in \{0,\sqrt{K}\}\}$$
and $(M^K_t)_{t\geq 0}$ is a martingale for any $K$.
The probability change associated with is given for $t\geq 0$ by
$$\widetilde{\mathbb P}_{\vert \mathcal F_t^K}= M_{t\wedge T^K}^K{\mathbb P}_{\vert \mathcal F_t},$$
where $(\mathcal F_t^K)_{t\geq 0}$ is the filtration associated with $N^K$.

\begin{proposition}
Under $\widetilde{\mathbb P}$, the process $(N^K_{t\wedge T^K})_{t\geq 0}$ is a càdlàg strong Markov process on $[0,\sqrt K]$, that we write $Y^K$ for convenience. It is characterized by
\begin{align}
\label{chgP}    
\E^y (F(Y^K_s : s\leq t))=\widetilde{\E}^y(F(N_{s\wedge T^K}^K : s\leq t))=\E^y(F(N_s^K : s\leq t) M_{t\wedge T^K}^K)
\end{align}
for any $y\in \mathbb N$,  $t\geq 0$ and $F:  \mathbb D([0,t], \R)\rightarrow \R $ non-negative and measurable. The semigroup
$P^Y$ of $Y^K$ is  given by
$$P^Y_tf(y)=\E^y(f(N_{t\wedge T^K}^K) M_{t\wedge T^K}^K)$$
for $y\in \mathbb N$ and $f$ bounded from $\mathbb N$ to $\mathbb R$.
Its infinitesimal generator is 
\begin{align}
Q^Kf(y)
&=b\left(\frac{y}{K}\right)(y+1)[f(y+1)-f(y)]+d\left(\frac{y}{K}\right) (y-1)[f(y-1)-f(y)] \label{defQK}
\end{align}
for $0<y<\sqrt{K}$ and  $Q^Kf(y)=0$ for $y\in \{0,\sqrt{K}\}.$
\end{proposition}
Let us identify the semigroup and generator of $Y^K$ and refer to e.g.   chapter VII in \cite{Sharpe1988}
for details on change of probability in our context.
\begin{proof} 
For convenience, we write
$$g(x)=xf(x), \qquad E_t^K =\frac{1}{N_0^K}\exp\left[-\int_0^t  \left( b\left(N^K_s/K\right)-d\left({N^K_s}/K\right)\right) ds\right],$$
so that
$$f(N_{t\wedge T^K}^K) M_{t\wedge T^K}^K=g(N_{t\wedge T^K}^K)E_t^K.$$
Recall the SDE representation \eqref{SDEdef} by Poisson point measure,
\begin{align*}
N_t^K&=N_0^K+\int_0^t\int_{\mathbb R_+} 1_{u\leq \lambda(N_s^K)}\pi^b(ds,du) -\int_0^t\int_{\mathbb R_+} 1_{u\leq \mu(N_s^K)}\mathcal \pi^d(ds,du),
\end{align*}
where for convenience 
we write
$$\lambda(n)=n\, b(n/K), \qquad \mu(n)=n\, d(n/K).$$
 We get
\begin{align*}
f(N_{t}^K) M_{t}^K=&g(N_{t}^K)E_t^K\\
=&f(N_0^K)+\int_0^t\int_{\mathbb R_+} 1_{u\leq \lambda(N_{s-}^K)}(g(N_{s-}^K+1)-g(N_{s-}^K))E_s^K \pi^b(ds,du)\\
&+\int_0^t\int_{\mathbb R_+} 1_{u\leq \mu(N_{s-}^K)}(g(N_{s-}^K-1)-g(N_{s-}^K))E_s^K\pi^d(ds,du)\\
& -\int_0^t g(N_s^K)\left(b\left(N^K_s/K)\right)-d({N^K_s}/K)\right)E_s^K ds
\end{align*}
Then, considering time $t\wedge T^K$ and compensating the Poisson point measure yields the semi martingale decomposition
\begin{align*}
&f(N_{t\wedge T^K}^K) M_{t\wedge T^K}^K =f(N_0^K)+\int_0^{t\wedge T^K} \mathcal A^Kf(N_s^K) \, E_s^K \, ds +\tilde M_t^K
\end{align*}
where $\tilde M^K$ is a martingale starting from $0$ and
$\mathcal A^K$ is defined by
\begin{align*}
\mathcal A^Kf(y)&= \lambda(y)(g(y+1)-g(y))+ \mu(y)(g(y-1)-g(y))-(b(y/K)-d(y/K)) g(y)\\
&=\lambda(y)(y+1)f(y+1)+\mu(y)(y-1)f(y-1)-\lambda(y)(y+1)f(y)-\mu(y)(y-1)f(y).
\end{align*}
Taking expectation
\begin{align*}
P^Y_tf(y)&=\E^y(f(N_{t\wedge T^K}^K) M_{t\wedge T^K}^K)=f(y)+\E^y \left(\int_0^{t\wedge T^K} \mathcal A^Kf(N_s^K) \, E_s^K \, ds \right).
\end{align*}
This allows to identify the law of Markov process $Y^K$. In particular we can differentiate at time $0$ and get the generator for any function $f$. By dominated convergence (note that localisation by $T^K$ bound all quantities in the integral), $t\rightarrow P^Y_tf$ is differentiable and
$$\frac{\partial}{\partial t} P^Y_tf(y)_{\vert t=0}= \mathcal A^Kf(y)/y=Q^Kf(y),$$
 where $Q^K$ is given by  \eqref{defQK}.
\end{proof}
We consider now the scaled process
$$W^K_t = K^{-1/2} Y^K_{t\,K^{1/2}}$$
for $t\geq 0$.  Recalling   expression \eqref{defQK} for the generator of $Y^K$,  the generator $L^K$ of $W^K$
satisfies: 
\begin{align}
L^Kf(w)&=K^{1/2}Q^K\left[f\left(\frac{.}{\sqrt{K}}\right)\right] (\sqrt{K}w)\nonumber \\
&=\sqrt{K}b\left(\frac{w}{\sqrt{K}}\right)(\sqrt{K}w+1)[f((\sqrt{K}w+1)/\sqrt{K})-f(w)]\nonumber\\
&\qquad +\sqrt{K}d\left(\frac{w}{\sqrt{K}}\right) (\sqrt{K}w-1)[f((\sqrt{K}w-1)/\sqrt{K})-f(w)]. \label{expWK}
\end{align}
Writing 
$$\tau_1^K=\inf\{ t\geq 0 : W^K_t=1\},$$
we can now prove the following result on the survival of $N^K$ and its trajectory, before $\tau_1^K$. We   use the changed of probability of the previous proposition and the
scaled process $W^K$, together with the scaled growth rate
$a_K$ defined by 
$$a_K(w)=\sqrt{K}(b(w/\sqrt{K})-d(w/\sqrt{K})).$$
\begin{lemma} \label{alacon}
For  $K$ large enough, 
$$\mathbb P\left( T^K_{\sqrt{K}}<T^K_0\right)= \frac{1}{\sqrt{K}}\E \left(
\exp\left[\int_0^{\tau_1^K} \! a_K(W_s^K) \, ds \right]\right).$$
and  for any  function~$F:\mathbb D([0,\infty),[0,\infty))\rightarrow\mathbb{R}$ continuous and bounded,
\begin{align*}
&\mathbb{E}\!\left[
F\left(\sqrt{K}^{-1}N^K_{t\sqrt{K}\wedge T^K} : t\geq 0\right) 1_{T^K_{\sqrt{K}}< T^K_0}
\right] \\
& \qquad \qquad \qquad = \frac{1}{\sqrt{K}} \mathbb{E}\!\left[
F\big(W^K_{t} : t\geq 0\big)\, \,
\exp\!\left(\int_0^{\tau_1^K} \!a_K(W_s^K) ds\right)
\right].
\end{align*}
Furthermore, as $K\rightarrow \infty$,  $a_K$ converges to $x\rightarrow a.x$ uniformly on $[0,1]$ and $W^K$ converges in law in $\mathbb D(\mathbb R_+, \mathbb R_+)$ to $W$ defined by $\eqref{eq:w}$.
\end{lemma}
\begin{proof}
Firstly, $T^K$ is finite a.s. for $K$ large enough. Indeed, using Assumption \ref{assu}, $b$ becomes positive on $]0,1/\sqrt{K}]$ for $K$ large enough and $\sqrt{K}$ is accessible.\\
Secondly, $M^K_{t\wedge T^K}$ is bounded. Indeed, since, locally around zero, $b(x)-d(x) = ax + O(x^{1+\alpha})\geq ax/2$, and for any $s\leq T^K$, $N^K_s/K\leq K^{-1/2}$, the exponential in~\eqref{eq:MK} is bounded by one.

Then we get, by letting $t\rightarrow \infty$ in~\eqref{chgP}, using dominated convergence,
\begin{align*}
\E^y (F(Y^K_{s\wedge S^K} : s\geq 0))=\E^y(F(N_{s\wedge T^K}^K : s\geq 0) M_{T^K}^K),
\end{align*}
where $F:\mathbb D([0,\infty),[0,\infty))\rightarrow\mathbb{R}$ is continuous and bounded and, recalling that $Y^K$ does not reach $0$,  $S^K$ is defined by
$$S^K=\inf\left\{t \geq 0 : Y^K_t=\sqrt{K}\right\}.$$
Using now 
$$F( x_s : s\geq 0)=G(x_s : s\geq 0)1_{\inf\{s \geq 0  \, : \, x_s=\sqrt{K}\} \, <\, \inf\{s \geq 0 \, : \,  x_s=\sqrt{K}\} }$$
and that 
$$M^K_{T^K}=\frac{\sqrt{K}}{N_0^K}\,\exp\left[-\int_0^{ T^K_{\sqrt{K}}} \! (b(N^K_s/K)-d(N^K_s/K)) \, ds\right],$$
a.s. on the event $\{T^K_{\sqrt{K}}<T^K_0\}$, we get
\begin{align*}
&\E^y (G(Y^K_{s\wedge \tau^K} : s\geq 0))\\
&\qquad =\frac{\sqrt{K}}{y}\E^y\left(G(N_{s\wedge T^K}^K : s\geq 0) 1_{T^K_{\sqrt{K}}<T^K_0} \exp\left[-\int_0^{T^K_{\sqrt{K}}} \! (b(N^K_s/K)-d(N^K_s/K)) \, ds\right]\right).
\end{align*}
It yields, with a suitable choice of $G$ that compensates the exponential term,
\begin{align*}
&\E^y\left(F(N_{s\wedge T^K}^K : s\geq 0)1_{T^K_{\sqrt{K}}<T^K_0}\right)\\
&\qquad =\frac{y}{\sqrt{K}}\E^y \left(F(Y^K_{s\wedge S^K}  : s\geq 0)
\exp\left[\int_0^{S^K} \! (b(Y^K_s/K)-d(Y^K_s/K)) \, ds \right]\right).
\end{align*}
It implies
\begin{align*}
&\mathbb{E}^y\!\left[F\left(\sqrt{K}^{-1}N^K_{t\sqrt{K}\wedge T^K} : t\geq 0\right)1_{T^K_{\sqrt{K}}< T^K_0}\right]\\
&\qquad \qquad =\frac{y}{\sqrt{K}}\mathbb{E}^y\!\left[F(W_s^K : s\geq 0)\exp\left[\sqrt{K}\int_0^{\tau_1^K} \left( b(W_s^K/\sqrt{K})-d(W_s^K/\sqrt{K}) \right)ds\right]\right]
\end{align*}
where we recall that $W^K=  Y^K_{t\,K^{1/2}}/\sqrt{K}$ 
for $t\geq 0$.

Finally, using  again Assumption \ref{assu}, $a_K$ converges to $x\mapsto a.x$ uniformly on $[0,1]$. Moreover the generator of $W^K$ 
 can be approximated as follows as $K$ goes to infinity
\begin{align*}
L^Kf(w)
&=2\tx f'(w)+aw^2f'(w)+w \tx f"(w)+o(1) 
\end{align*}
for  $w\in \{ n/ \sqrt{K} : n=1,\ldots, \sqrt{K}\}$.
Whence, the generator~$L^K$ of $W^K$ converges to the generator of the process~$W$ defined at~\eqref{eq:w}. In particular, from  diffusion approximation results of \cite{ethier2009markov} (e.g. Theorem~1.6.1 or Theorem~7.4.1),  process $W^K$ converges in distribution to $W$.
\end{proof}
To prove Theorem~\ref{thm:1}, we want now to use that $a_K$ and  $W^K$ converge respectively to  $a$ and $W$
to make
$$ \mathbb{E}\!\left[
F\big(W^K_{t} : t\geq 0\big)\, \,
\exp\!\left(\int_0^{\tau_1^K} \!a_K(W_s^K) ds\right)
\right]$$
converge
to
$$
\mathbb{E}\!\left[
F\big(W_{t} : t\geq 0\big)\, \,
\exp\!\left(\int_0^{\tau_1} \! a W_s ds\right)
\right].$$
Since the exponential term is unbounded, convergence is not guaranteed.  Uniform integrability
of the family of  exponential  functional  involved is therefore required. This is the purpose of the following lemma.
\begin{lemma} \label{estmomexpo}
There exists~$\beta>a$ such that
\begin{equation}\label{eq:ui}
    \underset{K}{\sup}~\esp{\exp\left(\beta\int_0^{\tau_1^K} W^K_sds\right)}<\infty.
\end{equation}
\end{lemma}
\begin{proof}
We work with 
$$x=\frac{i}{\sqrt{K}}$$
for $i\in\llbracket 1,\sqrt{K}\rrbracket$. Let
$M>0$ and 
$$\phi_K(i)=\mathbb{E}^x\left[\exp{\left\{\left(\beta\int_0^{\tau_1^K} W^K_sds\right)\wedge M\right\}}\right], \qquad \phi_K(\sqrt{K})=1,$$
with $\tau^K_1$ the hitting time of one by~$W^K$. 
We want to prove that
$$\underset{K,M}{\sup}~\phi_K(1)<\infty,$$
which is equivalent to \eqref{eq:ui}.
We introduce the rates associated with $W^K$
\begin{align*}
    \lambda_i &=\sqrt{K}b\left(\frac{i}{{K}}\right)\left(i+1\right)= \sqrt{K}b\left(\frac{x}{\sqrt{K}}\right)\left(x\sqrt{K}+1\right), \qquad \\ 
    \mu_i&=\sqrt{K}d\left(\frac{i}{{K}}\right)\left(i-1\right)= \sqrt{K}d\left(\frac{x}{\sqrt{K}}\right)\left(x\sqrt{K}-1\right).
\end{align*}
Notice that Assumption \ref{assu} guarantees
$$\lambda_i+\mu_i =2\tx \sqrt{K}i+O(i^2/\sqrt{K})= 2\tx\, x\, K + O(x^2 \sqrt{K}).$$
Then, introduce the normalized version of these rates: for $i\in\llbracket 1,\sqrt{K}\rrbracket$,
\begin{align*}
    \overline{\lambda}_i=& \frac{\lambda_{i}}{\lambda_{i} + \mu_{i}},\qquad 
    \overline{\mu}_i= \frac{\mu_{i}}{\lambda_{i} + \mu_{i}}.
\end{align*}
Writing $\sigma_1^K$ is the first jump time of~$W^K$, then
\begin{align*}
\phi_K(i) =& \mathbb{E}^x\left[\mathbb{E}\left[\left.\exp\left\{\left(\beta\int_0^{\sigma_1^K}W^K_sds +\beta\int_{\sigma_1^K}^{\tau^K_1} W^K_sds\right)\wedge M\right\}\right|\mathcal{F}_{\sigma_1^K}\right]\right]\\
\leq& \mathbb{E}^x\left[\mathbb{E}\left[\left.\exp\left\{\left(\beta\int_0^{\sigma_1^K}W^K_sds\right)\wedge  M\right\}\cdot \exp\left\{\left(\beta\int_{\sigma_1^K}^{\tau^K_1}W^K_sds\right)\wedge  M\right\}\right|\mathcal{F}_{\sigma_1^K}\right]\right]\\
\leq& \int_0^{+\infty} (\lambda_i+\mu_i)e^{-(\lambda_i+\mu_i)t + (\beta\,t\,x \wedge M)}dt\left[\bar\lambda_{i}\phi_K(i+1) + \bar\mu_{i} \phi_K(i-1))\right],
\end{align*}
where we have used the Markov property of $W^K$ at time $\sigma_1^K$ to obtain the last line, when the process either jumps from $i$ to  $i+1$ (birth) or to $i-1$ (death). The integral can be explicitly computed and gives the following family of linear system
\begin{align}
\phi_K({i})&\leq
[1+\eps_i]\left[\bar\lambda_{i}\phi_K({i+1}) + \bar\mu_{i} \phi_K({i-1})\right],\label{eq:recphi}
\end{align}
where 
\begin{align*}
\eps_i=\frac{\lambda_{i}+\mu_{i}}{\lambda_{i}+\mu_{i}-\beta\,i/\sqrt{K}}-1
=&\left(1- \frac{\beta}{2\tx {K} + O(i)}\right)^{-1}-1
= O(1/K),
\end{align*}
uniformly  for $i\leq\sqrt{K}$. In addition, $\eps_i>0$.
To study the sequence $(\phi_K(i))_i$, we use the sequences
$(\beta_i)_{i\geq 1}$ defined by
$$\beta_1=1/2, \quad  \beta_i=\frac{1}{z_i/(\beta_{i-1}+1)-1}, \quad \text{with } \quad 1/z_i=(1+\varepsilon_i)^2\overline{\mu}_i\overline{\lambda}_{i-1} $$
for $i\geq 2$ and $(\gamma_i)_{i\geq 1}$ defined by
$$\gamma_1=1,\qquad  \gamma_{i+1}=(\beta_i +1)(1+\varepsilon_i)\overline{\lambda}_i\gamma_i.$$
It implies for any $i\geq 1$, 
$$\gamma_i=\gamma_1\Pi_{j=1}^{i-1}\frac{\gamma_{j+1}}{\gamma_j}=\gamma_1 \Pi_{j=1}^{i-1}(\beta_j+1)(1+\varepsilon_j)\overline{\lambda}_j.$$
We observe that
\begin{equation}\label{eq:recgamma}
  \gamma_{i+1} = \beta_i \frac{\bar\lambda_i}{\bar\mu_i}\gamma_{i-1} \frac{1+\eps_{i-1}}{1+\eps_i} 
\end{equation}
and 
\begin{equation}
\label{telescop}
v_i=\gamma_{i+1}\phi_K(i+1)-\gamma_i \phi_K(i).
\end{equation}

{\it Step~1.} We prove that
\begin{equation}\label{eq:recv}
    v_i\geq \beta_i v_{i-1} + \phi_K(i-1)\beta_i \gamma_{i-1}\left(1-\frac{1+\eps_{i-1}}{1+\eps_i}\right).
\end{equation}
Indeed, using~\eqref{eq:recphi} and the definition of $\gamma_{i+1}$,
\begin{align}
    v_i = & \gamma_{i+1} \phi_K(i+1) - \gamma_i \phi_K(i)\nonumber\\
    \geq& \ \phi_K(i+1)\left[(\beta_{i}+1)(1+\eps_i)\bar\lambda_i\gamma_i - \bar\lambda_i\gamma_i(1+\eps_i)\right] - \gamma_i (1+\eps_i)\bar\mu_i\phi_K(i-1)\nonumber\\
    \geq& \phi_K(i+1)\beta_i (1+\eps_i)\bar\lambda_i\gamma_i - \gamma_i (1+\eps_i)\bar\mu_i\phi_K(i-1),\label{eq:viinter}
\end{align}

On the other hand, using~\eqref{eq:recphi},
$$v_{i-1} = \gamma_i\phi_K(i) - \gamma_{i-1}\phi_K(i-1) \leq \gamma_i(1+\eps_i)\bar\lambda_i\phi_K(i+1) + (\gamma_i(1+\eps_i)\bar\mu_i -\gamma_{i-1})\phi_K(i-1),$$
meaning
$$\gamma_i(1+\eps_i)\bar\lambda_i\phi_K(i+1)\geq v_{i-1} + (\gamma_{i-1}-\gamma_i(1+\eps_i)\bar\mu_i)\phi_K(i-1).$$
Using the previous inequality in~\eqref{eq:viinter} implies
\begin{align*}
    v_i \geq & \beta_i v_{i-1} + \beta_i (\gamma_{i-1}-\gamma_i(1+\eps_i)\bar\mu_i)\phi_K(i-1) - \gamma_i(1+\eps_i) \bar\mu_i\phi_K(i-1)\\
    \geq& \beta_i v_{i-1} + \beta_i\gamma_{i-1}\phi_K(i-1) -\phi_K(i-1)\bar\mu_i \gamma_i(1+\eps_i)(1+\beta_i).
\end{align*}
Using successively the definition of $\gamma_{i+1}$ and \eqref{eq:recgamma},
$$\bar\mu_i\gamma_i(1+\eps_i)(1+\beta_i) = \gamma_{i+1}\frac{\bar\mu_i}{\bar\lambda_i} = \beta_i\gamma_{i-1}\frac{1+\eps_i}{1+\eps_{i-1}}.$$
Hence
$$v_i\geq \beta_i v_{i-1} + \beta_i\gamma_{i-1}\phi_K(i-1) - \beta_i\gamma_{i-1}\frac{1+\eps_i}{1+\eps_{i-1}}\phi_K(i-1),$$
proving~\eqref{eq:recv}.
Then we define
$$\eta_K = \max_i \beta_i\gamma_{i-1}\left|1-\frac{1+\eps_i}{1+\eps_{i-1}}\right|$$
such that
\begin{equation}\label{eq:recv2}
v_i \geq \beta_i v_{i-1} - \phi_K(i-1) \eta_K.
\end{equation}

{\it Step~2.} Now we need to study the sequences $(\gamma_j)_{j\geq 1}$, $(\beta_j)_{j\geq 1}$ and control its successive products. It is achieved in Appendix, and we 
show in \eqref{encadrement} that for $i$ large enough,
\begin{equation} 
\label{encadrement2}
 -\frac{3}{2} \leq \beta_i -1\leq -\frac{1}{i/2-c''}.
 \end{equation} 
This allows to deal with
\begin{equation}
\label{dealgamma}
 \gamma_i = \Pi_{j=1}^{i-1} (1+\varepsilon_j) \Pi_{j=1}^{i-1}(\beta_j+1)\overline{\lambda}_j.   
\end{equation} 
Moreover 
$$\beta_j+1=2\left(1+\frac{\beta_j-1}{2}\right), \quad \overline{\lambda}_j=\frac{1+\overline{\lambda}_j-\overline{\mu}_j}{2}.$$
Then using   $\beta_i\leq 1$  (for $i$ large enough) and   the right hand side of \eqref{encadrement2} and $\overline{\lambda}_j-\overline{\mu}_j$ positive:
\begin{equation}\label{eq:lambdamu}
 \Pi_{j=1}^{i-1}(\beta_j+1)\overline{\lambda}_j= \Pi_{j=1}^{i-1} (1+(\beta_j-1)/2)\left(1+\overline{\lambda}_j-\overline{\mu}_j\right)\leq   c_1\Pi_{j=1}^{i-1} \left(1-1/j+\overline{\lambda}_j-\overline{\mu}_j\right)
\end{equation}
for $i\leq \sqrt{K}$ with $c_1$ some positive constant (independent of $K$ and $i$).
Besides, 
by definition of $\bar\lambda,\bar\mu$,
$$\bar\lambda_j - \bar\mu_j - \frac{1}{j} = \frac{1}{1+O(j/K)}\left(\frac{a}{2d}\cdot \frac{j}{K} + \frac{1}{j} + o\left(1/\sqrt{K}\right)\right) - \frac{1}{j}
\leq \frac{c_2}{\sqrt{K}}.$$
Plugging this in~\eqref{eq:lambdamu} yields ensures that 
$$\Pi_{j=1}^{i-1}(\beta_j+1)\overline{\lambda}_j\leq \left(1+\frac{c_2}{\sqrt{K}}\right)^{\sqrt K}$$ is bounded for  $i\leq \sqrt{K}$. Adding that 
$$\Pi_{j=1}^{i-1} (1+\varepsilon_j) \leq \left(1+\frac{c}{\sqrt{K}}\right)^{i-1}$$
is also bounded for $i\leq \sqrt{K}$, \eqref{dealgamma} implies   that $(\gamma_i)_{i\leq \sqrt{K}}$ is bounded (uniformly in $K$).\\

{\it Step~3.} Let us prove that $\eta_K = O(K^{-3/2})$, with $\eta_K$ from~\eqref{eq:recv2}. Since the sequences $(\beta_i)_i$ and $(\gamma_i)_i$ are bounded, it is a consequence of the fact that
$$|\eps_i - \eps_{i-1}| = O(K^{-3/2})$$
uniformly w.r.t.~$i$. Indeed, by definition
\begin{align*}
    \left|\eps_i - \eps_{i-1}\right|\leq& c K^{-3/2} \left|\frac{\lambda_i+\mu_i}{i} - \frac{\lambda_{i-1}+\mu_{i-1}}{i-1}\right|,
\end{align*}
and, from our assumptions,
$$\frac{\lambda_i+\mu_i}{i} = 2\rho\sqrt{K} + O(i/\sqrt{K}) + a/\sqrt{K} + O(i^{1+\alpha}/K^{\alpha+1/2}) = 2\rho\sqrt{K} + a/\sqrt{K} + O(1).$$

{\it Step~4.} We can now conclude the proof. Summing equation \eqref{eq:recv2} over $i=1\ldots \sqrt{K}-1$, 
\begin{align*}
    v_j \geq& \prod_{i=2}^j \beta_j v_1 - \eta_K\sum_{l=1}^{j-1}\phi_K(l-1)\prod_{i=l+2}^j \beta_i\\
    \sum_{j=1}^{\sqrt{K}} v_j\geq& \sum_{j=1}^{\sqrt{K}}\prod_{i=2}^j \beta_j v_1 - \eta_K\phi_K(1)\sum_{l=1}^{\sqrt{K}-1}\sum_{j=l+1}^{\sqrt{K}}\prod_{i=l+2}^{j}\beta_i.
\end{align*}
Since $\beta_i\leq 1$ (for $i$ large enough),
$$\sum_{j=1}^{\sqrt{K}} v_j\geq \sum_{j=1}^{\sqrt{K}}\prod_{i=2}^j \beta_j v_1 - \eta_K\phi_K(1) K \geq \sum_{j=1}^{\sqrt{K}}\prod_{i=2}^j \beta_j v_1 - C K^{-1/2} \phi_K(1).$$
Then, denoting
$$S_K = \sum_{j=1}^{\sqrt{K}}\prod_{i=2}^j \beta_j,$$
and noticing that
$$v_1 = \gamma_2\phi_K(2) - \phi_K(1) \geq \left[\frac{\gamma_2}{\bar\lambda_1(1+\eps_1)} - 1\right]\phi_K(1) = \beta_1 \phi_K(1) = \phi_K(1)/2,$$
we obtain
$$\gamma_{\sqrt{K}} - \phi_K(1) = \sum_{j=1}^{\sqrt{K}}v_j \geq \left[\frac12 S_K - CK^{-1/2}\right] \phi_K(1),$$
implying that
$$\phi_K(1) \leq \gamma_{\sqrt{K}}/\left(1 + S_K/2 - CK^{-1/2}\right),$$
which is bounded uniformly w.r.t.~$K$.
\end{proof}
We can now put pieces together to conclude.
\begin{proof}[Proof of Theorem~\ref{thm:1}]
Using Lemma \ref{alacon}, as $K\rightarrow \infty$,  $a_K$ converges to $x\mapsto a.x$ uniformly on $[0,1]$ and $W^K$ converges in law in $\mathbb D(\mathbb R_+, \mathbb R_+)$ to $W$ defined by $\eqref{eq:w}$. We get for $F$ bounded and continuous that
$$
F\big(W^K_{t} : t\geq 0\big)\, \,
\exp\!\left(\int_0^{\tau_1^K} \!a_K(W_s^K) ds\right)\Rightarrow 
F\big(W_{t} : t\geq 0\big)\, \,
\exp\!\left(\int_0^{\tau_1} \!a W_s ds\right)$$
in law as $K$ goes to $\infty$. To take expectation in this limit, we need to check uniform integrability of the left hand side. It is a consequence of exponential estimates of Lemma \ref{estmomexpo} and the fact that
$a_K(x)\leq \beta x$ for any $x\in [0,1/\sqrt{K}]$ and $K$ large enough. 
\end{proof}

\subsection{Intermediate phase : proof of Theorem~\ref{thm:regime22}}

Firstly, let us control the stopping time $\bar T_x(K^\eps,K^{-\eta})$ (with probability going to one as $K$ goes to infinity) with the following lemma.
\begin{lemma}\label{lem:controlT}
for all $\delta > \eta + 2\eps,$
$$\underset{K^{-\eta}\leq x\leq K^\eps}{\sup}\mathbb{P}\left(\bar T_x(K^\eps,K^{-\eta}) > K^\delta\right)\underset{K\rightarrow\infty}{\longrightarrow}0.$$
\end{lemma}
\begin{proof} Estimates of exit times for one dimensional diffusion are classical and come back from
 \cite{fellerfrontiere}. See 
also \cite{meleard2026} for a recent monograph with applications in life sciences. We set
\begin{align*}
Q(x) :=& -\int_0^x \frac{2a\,y^2}{2\tx\,y}dy = -\frac{a}{\tx} \int_0^x y\,dy = -\frac{a\, x^2}{2\tx},\\
g_K(x) :=& -\Phi_K(x) + \frac{\Psi_K(x)}{\Psi_K(K^{\eps})}\Phi_K(K^\eps),
\end{align*}
with
\begin{align*}
\Phi_K(x) :=& \int_{K^{-\eta}}^{x} e^{Q(y)}\left(\int_{K^{-\eta}}^y \frac{2 e^{-Q(z)}}{2\tx\, z}dz\right)dy,\\
\Psi_K(x) :=& \int_{K^{-\eta}}^{x} e^{Q(y)}dy.
\end{align*}


To control the function~$g_K$, let us notice that the function $x\mapsto Q(x)$ is non-increasing, and so
$$\Phi_K(x)\leq \int_{K^{-\eta}}^x e^{Q(y)} \frac1{\tx} K^\eta (y - K^{-\eta})e^{-Q(y)}dy \leq C\, K^{2\eps + \eta},$$
whence, using the fact that $\Psi_K$ is non-negative and non-increasing,
$$g_K(x)\leq \Phi_K(K^\eps) \leq C\, K^{2\eps + \eta}.$$

Then, by Markov's inequality,
$$\pro{\bar T_x(K^\eps,K^{-\eta}) > K^\delta}\leq K^{-\delta} \esp{\bar T_x(K^\eps,K^{-\eta})} = K^{-\delta} g_K(x) \leq C\, K^{2\eps + \eta-\delta},$$
where the identity for the mean hitting time can be found e.g. in Proposition 4.6.21 in \cite{meleard2026}. 
This entails the result.
\end{proof}

Let us introduce the processes $\zeta^K$ and $\zeta$, which are defined as $\xi^K,\bar\xi$ where the coefficients of the respective SDEs are "suitably truncated" outside $[K^{-\eta};K^\eps]$.  To define theses SDEs, let us introduce the function $\chi_K$ as a cutoff of the identity function.
%

\begin{lemma}\label{lem:rk}
There exists a sequence of $C^\infty$ and non-decreasing functions $(\chi_K)_K$ such that,
$$\chi_K(x) = \left\{\begin{array}{ll}\frac12 K^{-\eta}&\textrm{if }x\leq \frac12K^{-\eta},\\x&\textrm{if }x\in\left]K^{-\eta};K^\eps\right[,\\K^\eps + \frac12 K^{-\eta}&\textrm{if }x\geq K^\eps + \frac12K^{-\eta}.\end{array}\right.$$

In addition, there exists~$C>0$ such that, for all~$K\in\mathbb{N}^*$, and $1\leq k\leq 3$,
$$||\chi_K^{(k)}||_{\infty} \leq C\, K^{C\eta}.$$
\end{lemma}
This result being quite classical, its proof is omitted.

Now let $\bar\zeta^K$ be the solution of
$$d\bar\zeta^K_t = \mu_K(\bar\zeta^K_t)dt + \sigma_K(\bar\zeta^K_t)dB_t,$$
with
$$\mu_K(x) = a\, \chi_K(x)^2~~\textrm{ and }~~\sigma_K(x) = \sqrt{2\tx\, \chi_K(x)}.$$

And let us introduce a similar~$\zeta^K$ as a smooth truncated version of $\xi^K$. In the next proposition, we control the approximation of $\zeta^K$ by $\bar\zeta$.

\begin{proposition}\label{thm:regime2}
For all~$n\in\mathbb{N}^*, g_1,...,g_n\in C^3_b(\r)$, there exists~$C>0$ such that for any~$\delta,\eps,\eta>0,$
    \begin{equation}\label{eq:reg2}
    \underset{K^{-\eta}\leq x\leq K^\eps}{\sup}\,\,\underset{t\leq K^\delta}{\sup}\left|\mathbb{E}_{x}\left[g_1(\zeta^K_{t_1})...g_n(\zeta^K_{t_n}))\right] - \mathbb{E}_{x}\left[g_1(\bar\zeta^K_{t_1})...g_n(\bar \zeta^K_{t_n})\right]\right| \leq C K^{-\alpha/3}\, K^{C(\eps + \eta) + \delta}.
    \end{equation}
\end{proposition}

\begin{proof}
The techniques used in the proof are quite classical, so we just give the main steps of the proof with references.

{\it Step~1.} We control the difference between the generators of the two processes: for $g\in C^3_b(\r)$,
\begin{align*}
    A^{\bar\zeta^K}g(x)=& \mu_K(x) g'(x) + \frac12 \sigma_K(x)^2 g''(x),\\
    A^{\zeta^K}g(x)=& K \chi_K(x) \left\{b\left(\chi_K(x) K^{-1/2}\right)\left[g\left(x + K^{-1/2}\right) - g(x)\right]\right.\\&\left.\hspace*{1.5cm} + d\left(\chi_K(x) K^{-1/2}\right)\left[g\left(x - K^{-1/2}\right) - g(x)\right] \right\}.
\end{align*}

Denoting $\tilde x = \chi_K(x)$, we can write, by Taylor-Lagrange's inequality,
$$A^{\zeta^K}g(x)= \tilde x\sqrt{K} g'(x) (b-d)(\tilde x/\sqrt{K}) + \frac{\tilde x}{2}g''(x)(b+d)(\tilde x/\sqrt{K}) + ||g'''||_\infty O(K^{-1/2}).$$

Then, recalling our assumptions:
$$(b-d)(y) = a y + O(y^{1+\alpha})\textrm{ and }(b+d)(y) = 2\tx + O(y),$$

we obtain that, for all $K^{-\eta}\leq x\leq K^\eps$,
$$\left|A^{\bar\zeta^K}g(x) - A^{\zeta^K}g(x)\right| \leq C \left[||g'||_\infty \, O\left(\frac{\tilde x^{2+\alpha}}{K^{\alpha/2}}\right) + ||g''||_\infty \, O\left(\frac{\tilde x^2}{\sqrt{K}}\right) + ||g'''||_\infty \, O\left(K^{-1/2}\right)\right].$$

So, choosing $\eps$ small enough,
$$\left|A^{\bar\zeta^K}g(x) - A^{\zeta^K}g(x)\right| \leq C \left(||g'||_\infty+||g''||_\infty+||g'''||_\infty\right) K^{-\alpha/3}.$$

{\it Step~2.} We deduce an explicit convergence speed for the semigroups using the following Trotter-Kato formula (see the proof of Proposition~$B.2$ of \cite{erny2022mean}): for all $g\in C^3_b(\r)$, $t\geq 0,$ and $K^{-\eta}\leq x\leq K^\eps$,
$$P^{\bar\zeta^K}_tg(x) - P^{\zeta^K}_tg(x) = \int_0^t P^{\zeta^K}_{t-s}\left(A^{\bar\zeta^K}-A^{\zeta^K}\right)P^{\bar\zeta^K}_s g(x)ds.$$

From the formula above and the convergence speed of {\it Step~1},
\begin{align*}
    \left|P^{\bar\zeta^K}_tg(x) - P^{\zeta^K}_tg(x)\right|\leq& C\, K^{-\alpha/3} \int_0^t\left|\left|\left(P^{\bar\zeta^K}_s g\right)\right|\right|_{3,\infty}ds.
    \end{align*}
    
Since
    $$||\mu_K||_{3,\infty} + ||\sigma_K||_{3,\infty} \leq C\, K^{C(\eps + \eta)},$$
we can deduce (see e.g. the proof of Proposition~2.4 of \cite{erny2022mean}) that the derivatives of the semigroup of $\bar\zeta^K$ are bounded by 
    $$K^{C(\eps + \eta)},$$
    for some positive~$C>0.$

This entails
$$\underset{s\leq t}{\sup}\left|P^{\bar\zeta^K}_tg(x) - P^{\zeta^K}_tg(x)\right|\leq C\, t\, K^{-\alpha/3}K^{C(\eps + \eta)}||g||_{3,\infty}.$$

{\it Step~3.} Using Markov's property, we can deduce by induction that: for any~$n\in\mathbb{N}^*$, $t_1\leq t_2\leq ...\leq t_n \leq T,$ and $g_1,...,g_n\in C^3_b(\r),$
$$\left|\mathbb{E}_x\left[g_1(\bar\zeta^K_{t_1})...g_n(\bar\zeta^K_{t_n})\right]-\mathbb{E}_x\left[g_1(\zeta^K_{t_1})...g_n(\zeta^K_{t_n})\right]\right| \leq C_n\, T\, K^{-\alpha/3}\, K^{C_n(\eps + \eta)} ||g_1||_{3,\infty}...||g_n||_{3,\infty}.$$
This completes the proof of the proposition.
\end{proof}

Then, Theorem~\ref{thm:regime22} is a straightforward consequence of Lemma~\ref{lem:controlT} and Proposition~\ref{thm:regime2}.
\begin{proof}[Proof of Theorem~\ref{thm:regime22}]
Let us fix some $\eps,\delta,\eta>0$ will be chosen arbitrary small. Then, in order to control the following quantity
\begin{equation}\label{eq:quant2}
\underset{K^{-\eta}<x< K^\eps}{\sup}\,\,\underset{t_1\leq t_2\leq ...\leq t_n\leq \bar T^K_x}{\sup}\left|\mathbb{E}_x\left[g_1(\xi^K_{t_1})...g_n(\xi^K_{t_n})\right] - \mathbb{E}_x\left[g_1(\bar\xi_{t_1})...g_n(\bar\xi_{t_n})\right]\right|,
\end{equation}
let us split the probability space using the two complementary events
$$A^K_x:=\left\{\bar T_x(K^\eps,K^{-\eta})>K^\delta\right\}~~\textrm{ and }~~B^K_x:=\left\{\bar T_x(K^\eps,K^{-\eta})\leq K^\delta\right\}.$$

Thanks to Lemma~\ref{lem:controlT} (choosing $\delta>\eta+2\eps$),
$$\underset{K^{-\eta}\leq x\leq K^\eps}{\sup}\pro{A^K_x}\underset{K\to\infty}{\longrightarrow}0,$$
hence we only have to control~\eqref{eq:quant2} on the event $B^K_x$ (uniformly w.r.t. $x\in [K^{-\eta},K^\eps]$). And, on this event, \eqref{eq:quant2} is obviously upper-bounded by the LHS of~\eqref{eq:reg2}, which vanishes by Proposition~\ref{thm:regime2} if  we choose $\eps,\eta,\delta$ such that $C(\eps+\eta)+\delta<\alpha/3$ with $C$ is the positive constant of the RHs of~\eqref{eq:reg2}.
\end{proof}

\subsection{Entrance in macroscopic phase : proof of Theorem~\ref{thm:3}}

Let us introduce, for $0<x_0<v$ and $\eps>0,$
\begin{align*}
\tau(x_0,v) :=& \inf\left\{t\geq 0~:~x_K(t) = v\textrm{, with }x_K(0) := x_0\right\},\\
\theta^K_\eps(x_0) :=& \inf\left\{t\geq 0~:~\frac{X^K_t}{x_K(t)}> 1+\eps\right\}.
\end{align*}

In the following, we assume that the pair $(x_0,v)$ is chosen such that $v$ is reachable.

We use the stopping time $\theta^K_\eps$ to control the process~$X^K$ with a deterministic bound. The use of this stopping time is "free" thanks to the following lemma.

\begin{lemma}\label{lem:thetak}
Let $(T_K)_K$ be a sequence of times (that may be random). Assume that, there exist some~$\eps>0,  (x^K_0)_K$, such that, for all~$\eta>0$,
$$\pro{\underset{t\leq T_K \wedge \theta^K_\eps(x^K_0)}{\sup}\left|\frac{X^K_t}{x_K(t)} - 1\right|>\eta}\underset{K\rightarrow\infty}{\longrightarrow}0.$$

Then, for all~$\eta>0$,
$$\pro{\underset{t\leq T_K}{\sup}\left|\frac{X^K_t}{x_K(t)} - 1\right|>\eta}\underset{K\rightarrow\infty}{\longrightarrow}0.$$
\end{lemma}

\begin{proof}
For $T>0,\eta>0$, let us denote
$$E_K(T,\eta) := \left\{\underset{t\leq T}{\sup}\left|\frac{X^K_t}{x_K(t)} - 1\right|>\eta\right\},$$
and, for simplicity,
$$\theta^K := \theta^K_\eps(x^K_0).$$

Since
\begin{align*}
    E_K(T_K,\eta)\subseteq& \left(E_K(T_K,\eta) \cap \{\theta^K> T_K\}\right)\cup\left(E_K(T_K,\eta) \cap \{\theta^K\leq T_K\}\right)\\
    \subseteq& E_K(T_K\wedge \theta^K,\eta) \cup E_K(T_K\wedge \theta^K,\eps),
\end{align*}
the lemma is proved.
\end{proof}

The next result gives the key estimate to prove Theorem~\ref{thm:3}.

\begin{proposition}\label{prop:Xsurx}
Let $(x_K(0))_K,(v_K)_K$ be sequences of positive numbers, with $x_K(0)$ vanishing as $K$ goes to infinity, and $(v_K)_K$ be upper-bounded by~$v>0$ defined at Section~\ref{sec:regime3}. Assume that $x_K(0)$ belongs to~$\mathbb{N}/K$ and
$$X^K(0) = x_K(0).$$

Then, for all~$\eta>0$ and~$\eps>0$, 
$$\pro{\underset{t\leq \tau^K\wedge\theta^K}{\sup}\left|\frac{X^K_t}{x_K(t)} - 1\right|>\eta} \leq \frac1\eta C_{v} \frac{v_K^{1+\eps}}{K^{1/2} \, x_K(0)^{2+\eps}}$$
\end{proposition}

\begin{proof}
For the sake of clarity, we drop the notation~$K$ in the superscript for $X^K$ and in the subscript for~$x_K$. And we denote
$$\theta^K := \theta^K_\eps(x_K(0))\textrm{ and }\tau^K := \tau(x_K(0),v_K).$$

{\it Step~1.} The first step consists in proving that, for all~$\eta>0,$
\begin{equation}\label{eq:step1}
\pro{\underset{t\leq \tau^K\wedge\theta^K}{\sup}\left|\frac{X_t}{x(t)} - 1\right|>\eta} \leq C_{v,\eta} K^{-1/2} \left(\int_0^{\tau^K} \frac{1}{x(s)}ds\right)^{1/2} \exp\left\{(1+\eps)\kappa_v \int_0^{\tau^K} x(s)ds\right\}.
\end{equation}

By Ito's formula (or just by standard analysis, since $x_K$ is deterministic and $C^1$, and $X^K$ is a pure jump process with finite activity),
\begin{align*}
    \frac{X_t}{x(t)} =& 1 - \int_0^t \frac{X_s}{x(s)^2} d(x(s)) + \frac1K\int_{[0,t]\times\R_+} \frac1{x(s)}\uno{u\leq K X_{s-}b(X_{s-})}\pi^b(ds,du)\\
    &-\frac1K\int_{[0,t]\times\R_+} \frac1{x(s)}\uno{u\leq K X_{s-}d(X_{s-})}\pi^d(ds,du)\\
    =& 1 - \int_0^t \frac{X_s}{x(s)} (b-d)(x(s))ds + \int_0^t \frac{X_s}{x(s)} (b-d)(X_s)ds + M^K_t,
\end{align*}

with $M^K$ a martingale satisfying
$$\langle M^K\rangle_t = \frac1K \int_0^t \frac{X_s}{x(s)^2}(b+d)(X_s)ds.$$

Note in particular that
\begin{equation}\label{eq:mkt}
\underset{t\leq \tau^K\wedge\theta^K}{\sup}\langle M^K\rangle_t \leq C_{v} K^{-1}\int_0^{\tau^K} \frac{1}{x(s)}ds.
\end{equation}

Then, for $t\leq \tau^K\wedge\theta^K$,
\begin{align*}
    \frac{X_t}{x(t)} - 1 = & \int_0^t \frac{X_s}{x(s)} \left[(b-d)(X_s) - (b-d)(x(s))\right]ds + M^K_t\\
    \left|\frac{X_t}{x(t)} - 1\right|\leq& \kappa_v\int_0^t X_s \left|\frac{X_s}{x(s)} - 1\right|ds + |M^K_t|\\
    \leq& (1+\eps)a\int_0^t x(s) \left|\frac{X_s}{x(s)} - 1\right|ds + |M^K_t|.
\end{align*}

So, by Gr\"onwall's lemma, for all~$T\leq \tau^K\wedge\theta^K$,
$$\underset{t\leq T}{\sup}\left|\frac{X_t}{x(t)} - 1\right| \leq C_v \left(\underset{t\leq \tau^K\wedge\theta^K}{\sup}|M^K_t|\right) \exp\left\{(1+\eps)a\int_0^{T} x(s)ds\right\},$$
which gives
$$\underset{t\leq \tau^K\wedge\theta^K}{\sup}\left|\frac{X_t}{x(t)} - 1\right| \leq C_v \left(\underset{t\leq \tau^K\wedge\theta^K}{\sup}|M^K_t|\right) \exp\left\{(1+\eps)a\int_0^{\tau^K} x(s)ds\right\}.$$

Then, by Markov's inequality,
\begin{align*}
\pro{\underset{t\leq \tau^K\wedge\theta^K}{\sup}\left|\frac{X_t}{x(t)} - 1\right| > \eta} \leq& \pro{\underset{t\leq \tau^K\wedge\theta^K}{\sup}|M^K_t| > \frac{\eta}{C_v} \exp\left\{-(1+\eps)a\int_0^{\tau^K} x(s)ds\right\}}\\
\leq& C_{\eta,v} e^{(1+\eps)a\int_0^{\tau^K}x(s)ds}\esp{\underset{t\leq \tau^K\wedge\theta^K}{\sup}|M^K_t|}.
\end{align*}

Then, Burkholder-Davis-Gundy's inequality and~\eqref{eq:mkt} imply~\eqref{eq:step1}.

{\it Step~2.} Control of
\begin{equation}\label{eq:intxs}
    \int_0^{\tau^K} x(s)ds \leq C \ln(v_K/x_K(0)).
\end{equation}

Recall that, for $x\leq v$,
$$(b-d)(x)\geq \frac{a}{2}x,$$
whence for $t\in[0,\tau^K]$,
$$x(t) \leq \frac2{a} (b-d)(x(t)) = \frac2a \,\frac{x'(t)}{x(t)}.$$

In particular,
\begin{align*}
\int_0^{\tau^K} x(s)ds \leq & \frac2a\int_0^{\tau^K} \frac{x'(s)}{x(s)}ds = \frac2a \ln\left(\frac{x(\tau^K)}{x_K(0)}\right)\leq \frac2a \ln\left(\frac{v_K}{x_K(0)}\right).
\end{align*}

{\it Step~3.} Control of
\begin{equation}\label{eq:intinvxs}
    \int_0^{\tau^K} \frac1{x(s)}ds \leq \frac{2}{a \, x_K(0)^2}.
\end{equation}

Start from, for $t\leq \tau^K,$
$$\frac{x'(t)}{x(t)^2} = \frac{(b-d)(x(t))}{x(t)} \geq \frac{a}{2}.$$

Integrating the inequality above from 0 to $\tau^K$ leads to
$$\tau^K \leq \frac2a\left(\frac1{x_K(0)} - \frac1{v_K}\right) \leq \frac2{a\, x_K(0)}.$$

Since on $[0,\tau^K]$, $x'(t)$ is positive, the function $x$ is non-decreasing. So
$$\int_0^{\tau^K}\frac1{x(s)}ds \leq \frac{\tau^K}{x_K(0)} \leq \frac{2}{a\, x_K(0)^2}.$$

%

{\it Step~4.} Using~\eqref{eq:intxs} and~\eqref{eq:intinvxs} in~\eqref{eq:step1}, we obtain
\begin{align*}
    \pro{\underset{t\leq \tau^K\wedge\theta^K}{\sup}\left|\frac{X_t}{x(t)} - 1\right|>\eta}\leq& C_{v,\eta} K^{-1/2} \frac{1}{x_K(0)} \exp\left\{(1+\eps)\ln\left(\frac{v_K}{x_K(0)}\right)\right\}\\
    \leq& C_{v,\eta} \frac{v_K^{1+\eps}}{K^{1/2} \, x_K(0)^{2+\eps}},
\end{align*}
which concludes the proof.
\end{proof}

Then, the proof of Lemma~\ref{lem:341} follows easily:

\begin{proof}[Proof of Lemma~\ref{lem:341}]
It is a mere application of Proposition~\ref{prop:Xsurx} (with $v_K := v$ and choosing $\eps>0$ small enough such that $(1+\eps/2)\beta < 1/4$) and of Corollary~\ref{lem:thetak} (to remove the stopping $\theta^K_\eps(x_K(0))$ from the supremum).
\end{proof}

\begin{proof}[Proof of Theorem~\ref{thm:3}]
The idea of the proof is to apply Lemma~\ref{lem:341} a finite number of times successively. We define a sequence of exponents $(\gamma_n)_n$ defined in such way that: considering $v_K$ of order~$K^{-\gamma_n}$ and $x_K(0)$ of order~$K^{-\alpha}$ for any~$\alpha<\gamma_{n+1}$, we guarantee that
$$\underset{t\leq \tau(x_K(0),v_K)}{\sup}\left|\frac{X^K_t}{x_K(t)} - 1\right|\overset{\mathbb{P}}{\longrightarrow}0.$$

Thanks to Lemma~\ref{lem:341}, we know that
$$\gamma_0 = 0\textrm{ and }\gamma_1 = \frac14.$$

Next, by Proposition~\ref{prop:Xsurx}, we want to guarantee that, for some arbitrarily small~$\eps>0,$
$$K^{1/2} K^{(1+\eps)\gamma_n} K^{-(2+\eps)\alpha} \underset{K\rightarrow+\infty}{\longrightarrow}+\infty.$$

which implies that
$$(2+\eps)\alpha < \frac12 + (1+\eps)\gamma_n,$$
which can be guaranteed for an arbitrarily small~$\eps>0$ as long as
\begin{equation}\label{eq:alphagamman}
    \alpha < \frac14 + \frac12 \gamma_n.
\end{equation}

In particular, the critical choice leads to
$$\gamma_{n+1} = \frac14 + \frac12 \gamma_n,$$
which implies
$$\gamma_n = \frac12 - \frac{1}{2^{n+1}}\underset{n\rightarrow\infty}{\longrightarrow} \frac12.$$

Whence, if some~$\beta<1/2$ is fixed (with $x_K(0)\sim 1/K^\beta$), it is possible to choose some $n_0\in\mathbb{N}$ such that $\gamma_{n_0} > \beta$. Then, we can fix some~$\eps>0$ small enough such that
$$\beta + \eps < \gamma_{n_0} - \eps\textrm{ and }\eps < \frac{1}{2^{n_0 + 1}}.$$

Let us define,
\begin{align*}
&\delta_k := \gamma_k - \eps,~~~~1\leq k\leq n_0,\\
&\delta_0 := 0,\\
&x^{[k]}_K(0) = v^{[k-1]}_K(0),~~~~1\leq k\leq n_0-1,\\
&x^{[k]}_K(0)\sim K^{-\delta_{n_0 - k}},~~~~0\leq k\leq n_0 -1,\\
&v^{[n_0 - 1]}_K \sim 1 = K^{-\delta_0}.
\end{align*}

Notice that, since $\eps<1/2^{n_0+1}$, for all~$1\leq k\leq n_0$,
$$\gamma_{k-1} < \delta_k < \gamma_k,$$

and that, for all~$0\leq k\leq n_0-1,$
$$\delta_{k+1} < \frac14 + \frac12 \delta_k$$
such that, by~\eqref{eq:alphagamman} and Proposition~\ref{prop:Xsurx}, for each~$0\leq k\leq n_0-1$,
$$\underset{t\leq \tau(x^{[k]}_K(0),v^{[k]})}{\sup}\left|\frac{X^K_t}{x_K(t)} - 1\right|\overset{\mathbb{P}}{\longrightarrow}0.$$

Finally, writing
$$\left\{\underset{t\leq \tau(x_K(0),v)}{\sup}\left|\frac{X^K_t}{x_K(t)} - 1\right|>\eta\right\} = \bigcup_{k=0}^{n_0-1} \left\{\underset{t\leq \tau(x^{[k]}_K(0),v^{[k]})}{\sup}\left|\frac{X^K_t}{x_K(t)} - 1\right|>\eta\right\},$$
ends the proof.
\end{proof}

%
%

\appendix

\section{Proof of Lemma \ref{lem:extinction}}\label{sec:ext}


Consider the sequence $u_n=\mathbb{P}_n\Big(T^K_0 < T^K_M \Big),$ with $u_0 = 1$ and $u_M=0.$
By applying the law of total probability for the birth-death process, we obtain
   $$u_n=\frac{b(n/K)}{b(n/K)+d(n/K)}\mathbb{P}_{n+1}\Big(T^K_0 < T^K_M \Big)+\frac{d(n/K)}{b(n/K)+d(n/K)}\mathbb{P}_{n-1}\Big(T^K_0 < T^K_M \Big).$$
  So we get 
  $$b(n/K)(u_{n+1}-u_{n})=d(n/K)(u_n-u_{n-1}),$$
and after some calculation
\begin{equation*}
u_n = 1 - \frac{\sum\limits_{j=0}^{n-1} \prod\limits_{i=1}^{j} \frac{d(i/K)}{b(i/K)}}{\sum\limits_{j=0}^{M-1} \prod\limits_{i=1}^{j} \frac{d(i/K)}{b(i/K)}}.
\end{equation*}
Let
\begin{align*}
P_j = \prod\limits_{i=1}^{j} \frac{d(i/K)}{b(i/K)} = \left(\frac{(d + b-d)(i/K)}{d(i/K)}\right)^{-1}&=\left(1+\frac{a\, i/K + O((i/K)^{1+\alpha})}{\tx + O(i/K)}\right)^{-1}\\ &=\left(1+\frac{a}{\tx}\,\frac{i}{K} + O\left((i/K)^{1+\alpha}\right)\right)^{-1}.
\end{align*}
We have 
$$
\ln P_j = -\sum\limits_{i=1}^{j} \ln \left( 1+\frac{a}{\tx}\,\frac{i}{K} + O\left((i/K)^{1+\alpha}\right)\right)  
= -\sum\limits_{i=1}^{j} \left(\frac{a}{\tx} \frac{i}{K}+ R_i^K\right),  
$$
where $\lvert R_i^K \rvert \leq C (i/K)^{1+\alpha}.$

From this, let us prove item~$(i)$ of the lemma: so set $n_K = \lfloor K^\beta\rfloor$ (with $\beta<1/2$) and $M_K=\lfloor \sqrt{K}\rfloor$.
We obtain
$$-\frac{a}{\tx K} \frac{j(j+1)}{2}-\sum\limits_{i=1}^{j}C\left(\frac{i}{K}\right)^{1+\alpha}
\leq\ln P_j\leq-\frac{a}{\tx K} \frac{j(j+1)}{2}+\sum\limits_{i=1}^{j}C\left(\frac{i}{K}\right)^{1+\alpha}.$$
%
%
So $$\ln P_j\sim -\frac{a}{\tx K} \frac{j(j+1)}{2}.$$
Taking $\tilde R_j^K=C_1\frac{j^{2+\alpha}}{K^{1+\alpha}},$ for $j\leq \sqrt{K},\   \tilde R_j^K$ is bounded by $C_1K^{-\alpha/2}$.
$$-\frac{a}{\tx K} \frac{j(j+1)}{2}-\tilde R_j^K\leq\ln P_j\leq-\frac{a}{\tx K} \frac{j(j+1)}{2}+\tilde R_j^K$$
$$e^{-\frac{a}{\tx K} \frac{j(j+1)}{2}} e^{-\tilde R_j^K}\leq P_j\leq e^{-\frac{a}{\tx K} \frac{j(j+1)}{2}} e^{\tilde R_j^K}$$
$$e^{-\max\limits_{j} \tilde R_j^K}\sum\limits_{j=0}^{n_K-1} e^{-\frac{a}{\tx K} \frac{j(j+1)}{2}} \leq \sum\limits_{j=0}^{n_K-1} P_j\leq e^{\max\limits_{j}\tilde R_j^K}\sum\limits_{j=0}^{n_K-1} e^{-\frac{a}{\tx K} \frac{j(j+1)}{2}}. $$
Then, 
$$\sum\limits_{j=0}^{n_K-1} P_j\sim \sum\limits_{j=0}^{n_K-1} e^{-\frac{a}{\tx K} \frac{j^2}{2}}\textrm{ and }\sum\limits_{j=0}^{M_K-1} P_j\sim \sum\limits_{j=0}^{M_K-1} e^{-\frac{a}{\tx K} \frac{j^2}{2}},$$
for $n_K,M_K\leq \sqrt{K}.$
Then, since
$$  e^{-\frac{a {n^2_K}}{\tx K}}\leq e^{-\frac{a}{\tx K} \frac{j^2}{2}}\leq 1,$$
we obtain
\begin{equation*}
   n_Ke^{-\frac{a n^2_K}{\tx K}} \leq \sum\limits_{j=0}^{n_K-1} e^{-\frac{a}{\tx K} \frac{j^2}{2}}\leq n_K.
\end{equation*}
For $n_k=K^{\beta},$ for $\beta<1/2,$ 
\begin{equation*}
K^{\beta}e^{-\frac{a}{\tx}K^{2\beta -1}}\leq \sum\limits_{j=0}^{n_K-1}e^{-\frac{a}{\tx K} \frac{j^2}{2}}\leq K^{\beta}.
\end{equation*}
In addition, for $M_K=\sqrt{K}$ and using the Riemann sums approximation, we get 
\begin{equation*}
\sum\limits_{j=0}^{M_K-1}e^{-\frac{a}{\tx K} \frac{j^2}{2}}=\sqrt{K}\Big(\frac{1}{\sqrt{K}}\sum\limits_{j=0}^{M_K-1}e^{-\frac{a}{2\tx} (\frac{j}{\sqrt{K}})^2}\Big)\sim\sqrt{K}\int_0^1e^{-\frac{a}{2\tx} x^2}dx.
\end{equation*}
Therefore
\begin{equation*}
 \mathbb{P}_{\lfloor K^{\beta}\rfloor}\Big(T^K_0 > T^K_{\lfloor \sqrt{K}\rfloor} \Big)\sim\frac{\sum\limits_{j=0}^{n_K-1}e^{-\frac{a}{\tx K} \frac{j^2}{2}}}{\sum\limits_{j=0}^{M_K-1}e^{-\frac{a}{\tx K} \frac{j^2}{2}}}\sim  \frac{K^{\beta-1/2}}{\int_0^1e^{-\frac{a}{2\tx} x^2}dx} \to 0 \quad \text{as} \quad K \to \infty,
\end{equation*}
this proves item~$(i)$. Items~$(ii)$ and~$(iii)$ follow from similar calculations.

\section{Study of sequence $(\beta_j)_{j\geq 1}$}  
\label{studybeta}
We study the sequence $\beta$ by considering the sequence $w$:
$$w_i=\frac{1}{\beta_i-1}, \qquad \beta_i=1+\frac{1}{w_i}.$$
Recalling the expressions
$$\overline\mu_i = \frac{(\tx + O(i/K))(i-1)}{2\tx\, i + O(i^2/K)}~~;~~\bar\lambda_{i-1} = \frac{(\tx + O(i/K))i}{2\tx(i-1)+O(i^2/K)}~~;~~\frac1{z_i} = (1+ O(1/K))^2 \overline\mu_i\overline\lambda_{i-1},$$
we know that there exists~$c_0>0$ such that for all~$i$,
\begin{equation}\label{eq:zi}
\left|\frac{z_i}{4}-1\right|\leq c_0/K.
\end{equation}
By the definition of the previous quantities,
\begin{align*}
w_{i}-w_{i-1}&=-\frac{1}{2}+
\frac{(1+2w_{i-1})}{-2+\frac{1}{w_{i-1}(z_i/4-1)}}
=f_i^K(w_{i-1})
\end{align*}
Choose $c'$ (small enough) such that for any $K\geq 1$,
$$\frac{1+2c'\sqrt{K}}{-2+\frac{\sqrt{K}}{c_0c'}}\leq 1/4.$$
Then for any $i,K$, 
$$f_i^K([-c'\sqrt{K},c'\sqrt{K}])\subset (-\infty,-1/4]$$
This implies that, for any $i$ such that $\vert w_i\vert \leq c'\sqrt{K}$, $w_{i+1}=w_i-1/2+f_i^K(w_i)\leq w_i-1/4.$ Thus, before reaching the level $c'\sqrt{K}$ (in absolute value), the sequence $w_i$ decreases (at least) linearly with speed $1/4$. In  particular, there exists a constant $i_0$ such that for any $i\geq i_0$ $$ w_i \leq -2$$
Then, for $i\geq i_0$ (and $K$ large enough), we also obtain $1+2w_{i-1}\leq 0,$  $w_{i-1}(z_i/4-1)\leq 0$ and  $w_i-w_{i-1}\geq -1/2$. This ensures that 
$$-i/2 +c''\leq   w_i \leq -2.$$
Finally for $i$ large enough,
\begin{equation} 
\label{encadrement}
 -\frac{1}{2} \leq \beta_i \leq 1 +\frac{1}{c''-i/2}.
 \end{equation} 
{\bf Acknowledgments}
\begin{itemize}
\item Funded by INCa/Inserm research project in collaboration between Ecole Polytechnique and Saint-Louis Hospital.
\item Funded by the European Union (ERC, SINGER, 101054787). Views and opinions
expressed are however those of the author(s) only and do not necessarily reflect those of the European Union or the European Research Council. Neither the European Union nor the granting authority can be held responsible for them.
\item Funded by the Chair “Mod\'elisation Math\'ematique et Biodiversit\'e" of VEOLIA-Ecole Polytechnique-MNHN-F.X.
\end{itemize}

\bibliography{reference} 

@article{sagitov2015skeletons,
  title={Skeletons of near-critical Bienaym{\'e}-Galton-Watson branching processes},
  author={Sagitov, Serik and Serra, Maria Conceicao},
  journal={Advances in Applied Probability},
  volume={47},
  number={2},
  pages={530--544},
  year={2015},
  publisher={Cambridge University Press}
}

@book{denisov2025markov,
  title={Markov chains with asymptotically zero drift: Lamperti's problem},
  author={Denisov, Denis and Korshunov, Dmitry and Wachtel, Vitali},
  year={2025},
  publisher={Cambridge University Press}
}

@book{ikeda,
  title={Stochastic differential equations and diffusion processes},
  author={Ikeda, Nobuyuki and Watanabe, Shinzo},
  year={1989},
  publisher={North-Holland Publishing Company},
  edition={Second}
}

@article{bansaye2024sharp,
  title={Sharp approximation and hitting times for stochastic invasion processes},
  author={Bansaye, Vincent and Erny, Xavier and M{\'e}l{\'e}ard, Sylvie},
  journal={Stochastic Processes and their Applications},
  volume={178},
  pages={104458},
  year={2024},
  publisher={Elsevier}
}

@article{erny2022mean,
  title={Mean field limits for interacting Hawkes processes in a diffusive regime},
  author={Erny, Xavier and L{\"o}cherbach, Eva and Loukianova, Dasha},
  journal={Bernoulli},
  volume={28},
  number={1},
  pages={125--149},
  year={2022},
  publisher={Bernoulli Society for Mathematical Statistics and Probability}
}

@article{ball1995strong,
  title={Strong approximations for epidemic models},
  author={Ball, Frank and Donnelly, Peter},
  journal={Stochastic processes and their applications},
  volume={55},
  number={1},
  pages={1--21},
  year={1995},
  publisher={Elsevier}
}

@book{ethier2009markov,
  title={Markov processes: characterization and convergence},
  author={Ethier, Stewart N and Kurtz, Thomas G},
  year={2009},
  publisher={John Wiley \& Sons}
}

@book{athreya2012branching,
  title={Branching processes},
  author={Athreya, Krishna B and Ney, Peter E},
  year={2012},
  publisher={Springer Science \& Business Media}
}

@article{heathcote1967refinement,
  title={A refinement of two theorems in the theory of branching processes},
  author={Heathcote, Christopher R and Seneta, Eugene and Vere-Jones, David},
  journal={Theory of Probability \& Its Applications},
  volume={12},
  number={2},
  pages={297--301},
  year={1967},
  publisher={SIAM}
}

@article{seneta1966quasi,
  title={On quasi-stationary distributions in discrete-time Markov chains with a denumerable infinity of states},
  author={Seneta, Eugene and Vere-Jones, David},
  journal={Journal of Applied Probability},
  volume={3},
  number={2},
  pages={403--434},
  year={1966},
  publisher={Cambridge University Press}
}

@inproceedings{hering1977minimal,
  title={Minimal moment conditions in the limit theory for general Markov branching processes},
  author={Hering, Heinrich},
  booktitle={Annales de l'institut Henri Poincar{\'e}. Section B. Calcul des probabilit{\'e}s et statistiques},
  volume={13},
  number={4},
  pages={299--319},
  year={1977}
}

@article{aspandiiarov1996passage,
  title={Passage-time moments for nonnegative stochastic processes and an application to reflected random walks in a quadrant},
  author={Aspandiiarov, S and Iasnogorodski, Roudolf and Menshikov, M},
  journal={The Annals of Probability},
  volume={24},
  number={2},
  pages={932--960},
  year={1996},
  publisher={Institute of Mathematical Statistics}
}

@article{kersting2017recurrence,
  title={On recurrence and transience of multivariate near-critical stochastic processes},
  author={Kersting, G{\"o}tz},
  journal={Electronic Communications in Probability},
  volume={22},
  pages={1--12},
  year={2017}
}

@article{klebaner1989linear,
  title={Linear growth in near-critical population-size-dependent multitype Galton--Watson processes},
  author={Klebaner, Fima C},
  journal={Journal of applied probability},
  volume={26},
  number={3},
  pages={431--445},
  year={1989},
  publisher={Cambridge University Press}
}

@article{barbour2013approximating,
  title={Approximating the epidemic curve},
  author={Barbour, Andrew D and Reinert, Gesine},
  journal={Electronic Journal of Probability},
  volume={18},
  number={54},
  pages={1--30},
  year={2013},
  publisher={Institute of Mathematical Statistics and Bernoulli Society}
}

@article{barbour2015escape,
  title={Escape from the boundary in Markov population processes},
  author={Barbour, Andrew D and Hamza, Kais and Kaspi, Haya and Klebaner, Fima C},
  journal={Advances in Applied Probability},
  volume={47},
  number={4},
  pages={1190--1211},
  year={2015},
  publisher={Cambridge University Press}
}

@book{kurtz1981approximation,
  title     = {Approximation of Population Processes},
  author    = {Kurtz, Thomas G.},
  series    = {CBMS-NSF Regional Conference Series in Applied Mathematics},
  volume    = {36},
  year      = {1981},
  publisher = {SIAM}
}

@book{meleard2026,
  title={Random models in biology, ecology and evolution},
  author={M{\'e}l{\'e}ard, Sylvie},
  series={Texts in Applied Mathematics},
  year={2016},
  publisher={Springer Berlin, Heidelberg},
  isbn={978-3-662-73483-4},
}

@book{Sharpe1988,
  author    = {Michael Sharpe},
  title     = {General Theory of Markov Processes},
  series     = {Pure and Applied Mathematics},
  volume     = {133},
  publisher  = {Academic Press},
  address    = {Boston},
  year       = {1988},
  isbn       = {0-12-639060-6}
}

@article{champagnat2006microscopic,
  title={A microscopic interpretation for adaptive dynamics trait substitution sequence models},
  author={Champagnat, Nicolas},
  journal={Stochastic processes and their applications},
  volume={116},
  number={8},
  pages={1127--1160},
  year={2006},
  publisher={Elsevier}
}

@article{champagnat2011polymorphic,
  title={Polymorphic evolution sequence and evolutionary branching},
  author={Champagnat, Nicolas and M{\'e}l{\'e}ard, Sylvie},
  journal={Probability Theory and Related Fields},
  volume={151},
  number={1},
  pages={45--94},
  year={2011},
  publisher={Springer}
}

@article{fellerfrontiere,
  title={Diffusions processes in one dimension},
  author={Feller, William},
  journal={Transactions of the American Mathematical Society},
  volume={77},
  number={1},
  pages={1--31},
  year={1954},
  publisher={American Mathematical Society}
}

\end{document}